\documentclass[10pt,twoside]{amsart}

\usepackage{amsmath,amssymb
}
\usepackage{amsfonts}
\usepackage{amsthm,amscd}
\usepackage{amsmath}
\usepackage{amssymb}
\usepackage{amstext}
\usepackage{color}
\usepackage{latexsym}
\usepackage{amscd,graphics}
\usepackage{bbm}
\usepackage{empheq}
\usepackage{mathtools}
\usepackage{needspace}
\usepackage{enumitem}

\usepackage{geometry}
\PassOptionsToPackage{hyphens}{url}
\usepackage{hyperref}
\hypersetup{colorlinks = false, urlcolor = blue, bookmarksopen = true}

\newcommand*\diff{\mathop{}\!\mathrm{d}}
\newcommand*\Lip{\mathop{} \mathrm{Lip}}
\def\indep{\perp\!\!\!\perp}
\def\L2{\mathop{\rm L^2(\mu)}\nolimits}

\def\D{\mathrm {d}}

\newcommand{\comments}[1]{} 

\newtheorem{proposition}{Proposition}[section]
\newtheorem{lemma}[proposition]{Lemma}

\newtheorem{theorem}[proposition]{Theorem}

\newtheorem{corollary}[proposition]{Corollary}

\theoremstyle{remark}
\newtheorem{remark}[proposition]{Remark}
\theoremstyle{definition}

\numberwithin{equation}{section}

\begin{document}
\title[GIET from smooth stable foliation of a Derived from pseudo-Anosov map]{Smooth Generalized Interval Exchange Transformations with Wandering Intervals,\\ from explicit Derived from pseudo-Anosov maps}
\author{J{\'e}r{\^o}me Carrand
}
\address{Laboratoire de Probabilit\'es, Statistiques et Mod\'elisation (LPSM),  
CNRS, Sorbonne Universit\'e, Universit\'e de Paris,
4, Place Jussieu, 75005 Paris, France}
\email{jcarrand@lpsm.paris}

\date{\today}
\begin{abstract}
Starting from any pseudo-Anosov map $\varphi$ on a surface of genus $g \geqslant 2$, we construct explicitly a family of Derived from pseudo-Anosov maps $f$ by adapting the construction of Smale's Derived from Anosov maps on the two-torus. This is done by perturbing $\varphi$ at some fixed points. We first consider perturbations at every conical fixed point and then at regular fixed points. We establish the existence of a measure $\mu$, supported by the non-trivial unique minimal component of the stable foliation of $f$, with respect to which $f$ is mixing. In the process, we construct a uniquely ergodic Generalized Interval Exchange Transformation with a wandering interval that is semi-conjugated to a self-similar Interval Exchange Transformation. This Generalized Interval Exchange Transformation is obtained as the Poincar\'e map of a flow renormalized by $f$ which parametrizes stable foliation. When $f$ is $\mathcal{C}^2$, the flow and the Generalized Interval Exchange Transformation are~$\mathcal{C}^1$.
\end{abstract}
\thanks{This work was carried out during a one-year internship under the supervision of Corinna Ulcigrai.
Research supported by Institut f{\"u}r Mathematik,
Universit{\"a}t Z{\"u}rich, Switzerland and by the European Research Council (ERC)
under the European Union's Horizon 2020 research and innovation programme (grant agreement No 787304).}
\maketitle

\section{Introduction}

Since the work of Denjoy \cite{herman1979conjugaison, athanassopoulos2015denjoy}, it is known that every $\mathcal{C}^1$ diffeomorphism of the circle such that the logarithm of its derivative is a function of bounded variation has no wandering interval. There is no analogous result concerning interval exchange transformations. An interval exchange transformation -- IET for short -- is a piece-wise translation bijection, with finitely many branches, of a given base interval, while a generalized interval exchange transformation -- GIET for short -- is a bijection of the interval which is a piece-wise increasing homeomorphism with finitely many branches. These transformations can be seen as generalizations of, respectively, rigid translations and diffeomorphisms of the circle. See for instance the surveys \cite{forni2013introduction,yoccoz2010survey,zorich:hal-00453397}. On the other hand, IET and GIET can also be seen as the first return map of a flow on a surface to an interval. This is the point of view we will adopt.

In fact, there are several counter-examples, including very smooth ones. In \cite{levitt1987feuilletages} Levitt found an example of non-uniquely ergodic affine interval exchange transformation -- AIET for short -- with wandering intervals. Latter, using Rauzy--Veech induction, Camelier and Gutierrez \cite{camelier1997aiet} exhibited a uniquely ergodic AIET with wandering intervals, semi-conjugated to a self-similar IET -- \textit{i.e.} an IET induced by the foliation of a pseudo-Anosov diffeomorphism. Then Bressaud, Hubert and Maass \cite{hubert2010persistence} found a \emph{Galois type} criterion on eigenvalues of a matrix associated to a self-similar IET in order to admit a semi-conjugated AIET with wandering intervals. Finally, Marmi, Moussa and Yoccoz \cite{MMY2010aiet} proved that almost every IET admits a semi-conjugated AIET with a wandering interval.

In this paper, we prove the following result using an explicit construction.
\begin{theorem}\label{thm:main_1}
For all self-similar IET $T_0$, there exists a $\mathcal{C}^1$ GIET $T$ semi-conjugated to $T_0$ such that
\begin{enumerate}[label=(\roman*) ]
\item $T$ has a unique minimal set $\Omega$. This set is a Cantor set and is an attractor for $T$ and $T^{-1}$,
\item $T$ is uniquely ergodic, of unique invariant measure $\nu$ supported by $\Omega$,
\item $T$ has wandering intervals.
\end{enumerate}
Furthermore $T$ can be chosen to be the Poincar\'e map of a $\mathcal{C}^1$ flow of a surface and to have any number of wandering intervals.
\end{theorem}

The above theorem will follow from Proposition~\ref{prop:semiconjugacy}, Theorems~\ref{thm:T_is_minimal} and \ref{thm:omega_attractor}, and Proposition~\ref{prop:Omega_K_cap_gamma}.

The proof relies on a geometric construction initiated by Smale \cite[Section I.9]{smale1967differentiable}. More precisely, we built a transformation of $S_g$, the surface of genus $g$, by perturbing a pseudo-Anosov homeomorphism. We call it derived from pseudo-Anosov, as in \cite{Rodriguez06expansive, Barge11classification}. A pseudo-Anosov map on a surface $S_g$ of genus $g$ can be defined as an element of the homotopy class of a map preserving a flat metric on $S_g$ and locally given (in the natural coordinates associated to the half-translation structure) by the action of a diagonal and hyperbolic matrix of determinant $1$. Furthermore, up to sign, the matrix is constant on $S_g \smallsetminus \Sigma$, where $\Sigma$ denotes the set of conical points. Iterates of a pseudo-Anosov map are also pseudo-Anosov maps -- see \cite{lanneau2017tell} for equivalent definitions. 

The method used is similar to the one to pass from an Anosov map to a derived from Anosov diffeomorphism \cite{smale1967differentiable, Williams70DA, katok1997introduction, coudene2013book} and has already appeared in the literature \cite{Bonatti1998, Rodriguez06expansive, Barge11classification}. That is, we convert a fixed point of a pseudo-Anosov map, either regular or conical, into an attracting fixed point by a perturbation. In fact, in order to prove that the GIET in Theorem~\ref{thm:main_1} is piecewise $C^1$, we give an explicit construction of such a map by generalizing Coud\`ene's one for derived from Anosov maps.

For a large class of parameters, the family of maps obtained are derived from pseudo-Anosov and admits Axiom A attractors. It turns out that the stable manifolds of the constructed derived from pseudo-Anosov $f$ map can be parametrized by a $C^1$ flow $h_t$. This flow can be renormalized by $f$, in a similar fashion Giulietti--Liverani horocyclic flows are renormalized by an Anosov map \cite{GL_2019} -- see also \cite{Butterley2020parabolic} where a parabolic flow is renormalized by a partially hyperbolic map.

\begin{theorem}\label{thm:main_2}
For every Derived from pseudo-Anosov map $f$ constructed as in Section~\ref{sect:first_def_perturb}, there exists a hyperbolic attractor $K$ and a flow $h_t$ on $S_g \smallsetminus \Sigma$ such that
\begin{enumerate}[label=(\roman*) ]
\item $h_t$ is complete on $K$ and $\left. \frac{\mathrm{d}}{\mathrm{d}t}\right|_{t=0} h_t\vert_K$ spans the stable foliation of $f$,
\item $f \circ h_{\lambda t} = h_t \circ f$, where $\lambda >1$ is the expansion factor of the pseudo-Anosov homotopic to $f$,
\item $h_t$ is uniquely ergodic, with unique invariant measure $\mu$ supported by $K$,
\item $K$ is an attractor for future and past for $h_t$, on which $h_t$ is minimal.
\end{enumerate}
The flow $h_t$ and the map $T$ from Theorem~\ref{thm:main_1} are related as follow:
\begin{enumerate}[label=(\roman*) ]
\setcounter{enumi}{4}
\item $h_t$ is the suspension flow of the GIET $T$,
\item The unique invariant measures of $h_t$ and $T$ are related by $\mu = \nu \otimes \lambda$, where $\lambda$ is the Lebesgue measure. Also, $\mu$ is the SRB measure of $f$, for which $f$ is mixing -- results for Axiom A attractors from \cite{Ruelle1976measure} apply.
\end{enumerate}
Furthermore, we can construct $f$ such that $h_t$ is $C^1$.
\end{theorem}

The above theorem will follow from Proposition~\ref{prop:commutation_relation_and_stability_of_K}, Theorem~\ref{thm:unique_ergodicité_(h_t)}, Corollary~\ref{corol:ht_minimal}, Lemma~\ref{lemma:factorization_measure} and Theorem~\ref{thm:SRB_measure}.

\subsection{Organisation of the Paper} The paper is organised as follow: Section~\ref{sect:construction_and_attractor} is devoted to the construction of derived from pseudo-Anosov maps as in Theorem~\ref{thm:main_2}. First we recall the basic ideas of the construction, already present in the works \cite{Bonatti1998, Rodriguez06expansive, Barge11classification}. Then we give an explicit formulation for the map, generalizing Coud\`ene's construction of derived from Anosov map \cite{coudene2013book}. We prove that for appropriate parameters, the constructed map $f$ is indeed derived from pseudo-Anosov. Moreover, we prove the existence of an invariant compact set $K$, and show in Theorem~\ref{thm:K_is_hyperbolic} that this set is hyperbolic by computing explicitly a vector field $v^s$ spanning the stable foliation -- the unstable foliation coincide with the one of the initial pseudo-Anosov map $\varphi$. By construction, $v^s$ satisfies, for all $x$ in $K$, the relation 
\begin{align}\label{eq:commut_diff}
\D_x f \, v^s(x) = \lambda^{-1} v^s(f(x))
\end{align}
where $\lambda > 1$ is the expansion factor of $\varphi$.

Most of the work is carried in Section~\ref{sect:technicalities}. We prove in Theorems~\ref{thm:vs_bounded_continuous} and \ref{thm:vs_lipschitz} that $v^s$ can be extended over $S_g \smallsetminus \Sigma$ into a Lipschitz continuous vector field satisfying (\ref{eq:commut_diff}). Under a stronger assumption occurring in the construction of $f$, we prove that the extension of $v^s$ is $C^1$ (Theorem~\ref{thm:vs_C1}). We also prove that the homotopy between $\varphi$ and $f$ induces a homotopy between $v^s$ and the constant vector field spanning the stable foliation of $\varphi$ (Theorem~\ref{thm:regularity_beta_vs}). By integration of $v^s$, we get a flow $h_t$ that satisfies 
\begin{align}\label{eq:commut_flow}
f \circ h_t(x) = h_{\lambda^{-1}t} \circ f(x)
\end{align} because of (\ref{eq:commut_diff}), for all $x$ in $S_g \smallsetminus \Sigma$ and $t$ in $\mathbbm{R}$ whenever both sides are well defined. Using this commutation relation we deduce that $K$ is connected (Theorem~\ref{thm:K_connected}) and that $f:K \to K$ is topologically transitive (Theorem~\ref{thm:f_transitive_on_K}). A similar commutation relation is used by Butterley--Simonelli \cite{Butterley2020parabolic} where a parabolic flow is renormalized by a partially hyperbolic map on some 3-dimensional manifold.

In Section~\ref{sect:f_is_mixing} we consider the GIET $T$ obtained as the Poincar\'e map of $h_t$ to some transversal interval. Using the homotopy between vector fields, we prove that $T$ follows the same full path as a self-similar IET $T_0$ during the Rauzy--Veech algorithm, hence $T$ is semi-conjugated to $T_0$ by a result of Yoccoz \cite[Proposition 7]{yoccoz2005echanges} -- this proves Theorem~\ref{thm:main_1}. Unique ergodicity of $h_t$ then follows by writing $h_t$ as the suspension flow of $T$. Because of the commutation relation (\ref{eq:commut_flow}) and the usual functional characterization of mixing, $f$ is mixing with respect to the unique invariant measure of $h_t$.

In Section~\ref{sect:perturbation_regular_point}, we state the analogous results of Sections~\ref{sect:technicalities} and \ref{sect:f_is_mixing} in the case where the perturbation of the pseudo-Anosov map done in Section~\ref{sect:construction_and_attractor} is performed at a regular point instead of a conical point.

Finally, in the last section, using extensively Ruelle's results \cite{Ruelle1976measure} on the SRB measure of Axiom A attractors, we prove that $\mu$ is the unique SRB measure of $f^{-1}$ for a $\mathcal{C}^2$ perturbation, and that the correlations decrease exponentially fast for $\mathcal{C}^1$ observables. We also ask whether the result on Ruelle spectrum of linear pseudo-Anosov maps \cite{faure2019ruelle} extends to the present case, and if the asymptotic expansion of the ergodic integrals \cite[Corollary 1.5]{Forni2020equidistribution} applies for $h_t$.

\section{Derived from pseudo-Anosov map with smooth explicit foliations}\label{sect:construction_and_attractor}

The construction of derived from Anosov maps was initiated by Smale \cite{smale1967differentiable} by blowing up the stable manifold of a fixed point of an Anosov map. More precisely, an Anosov map is perturbed in such a way that some hyperbolic fixed point is turn into a sink (or a source). Derived from Anosov transformations are an example of Smale's diffeomorphism. The adaptation of this procedure to the setting of pseudo-Anosov maps was already known since the earliest works on pseudo-Anosov transformations \cite{fathi1979travaux,FLP_english}. General Smale diffeomorphisms of surfaces have been extensively studied by Bonatti and Langevin \cite{Bonatti1998}. These maps can in fact be \emph{blown down} into pseudo-Anosov ones as in \cite[Theorem 8.3.1]{Bonatti1998} in the sense that in some neighbourhood of the non-wandering set, the map is semi-conjugated to a pseudo-Anosov transformation.

The non-wandering set of a derived from pseudo-Anosov transformation gives an example of a non trivial -- different from a single periodic orbit -- attractor on a surface. In fact, Barge and Martensen proved \cite[Theorem 1]{Barge11classification} -- completing the work of \cite{Rodriguez06expansive} -- that any expansive and transitive attractor, different from a single periodic orbit, comes from a derived from pseudo-Anosov transformation.

Here, since we focus on the smoothness of the stable foliation in order to derive a smooth GIET as in Theorem~\ref{thm:main_1}, we give an explicit construction of derived from pseudo-Anosov maps, adapting Coud\`ene's one \cite[Chapter 9]{coudene2013book} to the setting of surface of genus larger than two.

In this section, we describe the explicit construction of a family of derived from pseudo-Anosov maps by perturbing a pseudo-Anosov transformation at each conical fixed points -- a similar procedure for regular fixed point is performed in Section~\ref{sect:perturbation_regular_point}. We prove that these maps are well defined, are homeomorphisms on $S_g$ and $C^1$ away from conical points, and that for a good choice of parameters, conical points are the only attractive fixed points. By connectedness of $S_g$, the complement $K$ of the union of basins of attraction is not empty. We prove in Theorem~\ref{thm:K_is_hyperbolic} that $K$ is hyperbolic by computing vector fields spanning the stable and the unstable foliations.

\subsection{Perturbation of a pseudo-Anosov}\label{sect:first_def_perturb}

Let $\varphi$ be a pseudo-Anosov transformation on the Riemann surface $S_g$ of genus $g$. Therefore the invariant foliations of $\varphi$ can be derived from a holomorphic quadratic differential $q$ invariant by $\varphi$. Up to consider a cover of order two in most cases, it is not too restrictive to assume that the quadratic differential is Abelian, in other words $q=\omega^2$ so that the transition maps, of the half-translation structure induced by natural coordinates of $\omega$, are translations. Up to multiplying  $\omega$ by a modulus one complex number, the horizontal and vertical foliations $\{\Re ( \omega) = 0 \}$ and $\{ \Im ( \omega ) = 0 \}$ are the invariant foliations of $\varphi$. Let $\lambda >1$ denote the stretch factor of $\varphi$. This stretching is assumed to correspond to the horizontal measured foliation. The vertical measured foliation is stretch by a factor $\lambda^{-1}$. Let $\Sigma$ be the set of points where $\omega$ vanishes, and we call these points \emph{conical points}. We now consider the flat structure induced by $\omega$ on $S_g \smallsetminus \Sigma$, that is charts $z$ so that $\omega =\mathrm{d} z$. In the neighbourhood of every conical point $\sigma \in \Sigma$, there exist an integer $n_{\sigma} >1$, an open set and a chart $z$ on this set such that $\omega = z^{n_{\sigma} -1} \D z$. The angle around $\sigma$ is then $2\pi n_{\sigma}$.

Outside of these neighbourhoods of points of $\Sigma$, we set $f$ to be equal to $\varphi$. We now construct $f$ to be a perturbation of $\varphi$ around each $\sigma$ in $\Sigma$. 

Let $\sigma$ be a conical point, $V_{\sigma}$ a neighbourhood of $\sigma$ and a chart $z$ on $V_{\sigma}$ so that $\omega = z^{n_{\sigma}-1} \D z$. Let $\xi$ be the branched cover at $\sigma$ associated to the chart $z$, $\xi : z \in z^{-1}V_{\sigma} \mapsto z^{n_{\sigma}} \in \xi(z^{-1}(V_{\sigma})) \subset \mathbbm{C}$. Let $(W_i)_{1 \leqslant i \leqslant 2n_{\sigma}}$ be a family of open sets of $\mathbbm{C} \smallsetminus \mathbbm{R}_+$ such that all $\xi|_{W_i}$ are homeomorphisms. Up to replacing $\varphi$ by one of its power, we assume that every conical point is fixed by $\varphi$ and that $\varphi$ respects the leaves of the branched covers: $\varphi(W_i) \cap V_{\sigma} \subset W_i$ for all $i$.

We can define $f$ on the base of the branched cover in the exact same manner as Smale \cite[Section I.9]{smale1967differentiable} does. In order to perform further analysis on the map, we give the following explicit formula that generalizes the one used in \cite[Chapter 9]{coudene2013book} and \cite{coudene2006pictures} in the case of the cat map on the two-torus.

For\footnote{Often in the rest of the paper, $x$ and $y$ might designate points in $S_g$. The notation $z=x+iy$ will only be used when clearly indicated.} $z=x+iy \in \mathbbm{C} \smallsetminus \mathbbm{R}_+$ in the image of $\xi$, we define $f$ as :
\begin{align*}
f(\xi|^{-1}_{W_i}(z)) \coloneqq \xi|^{-1}_{W_i} \left( (\lambda + \beta_{\sigma} k_{\sigma} (|z|/\alpha_{\sigma}))x + i \lambda^{-1}y \right),
\end{align*}
for some $\alpha_{\sigma} >0$, $\beta_{\sigma} \leqslant 0$ (we will be interested in the case $\beta < 1 - \lambda$) and with $|z|\leqslant \alpha_{\sigma}$ and where $k_{\sigma} : \mathbbm{R} \to \mathbbm{R}$ is an even unimodal map of class $\mathcal{C}^1$, compactly supported in $[-1,1]$ 
and such that $k'_{\sigma}$ is Lipschitz continuous, for example $k_{\sigma}(r) = (1 - r^2)^2 \mathbbm{1}_{[-1,1]}$. We do this perturbation at every conical point. We will see that such $f$ is well defined for small enough $\alpha_{\sigma}$.

When such a map $f$ is well defined, we will see in next section that interpolating $(\beta_{\sigma})_{\sigma \in \Sigma}$ with $0$ gives a homotopy between $f$ and $\varphi$. Therefore, $f$ is an example of derived from pseudo-Anosov transformation. 

We give in Figure~\ref{fig:saddles} a heuristic representation, when $n_{\sigma} = 1$ -- which corresponds to the case treated by Smale in \cite{smale1967differentiable}.

\begin{figure}
\begin{center}
\raisebox{-0.5\height}{\includegraphics[height=.2\textwidth , width=0.4\textwidth]{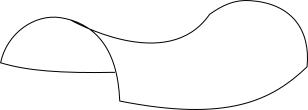}}\hspace{10mm}
\raisebox{-0.5\height}{\includegraphics[height=.2\textwidth , width=0.4\textwidth]{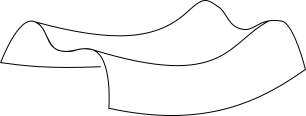}}
\end{center}
\caption{\label{fig:saddles} Heuristic representations of a saddle and of a saddle perturbed into a sink.}
\end{figure}

\begin{remark}
Because this construction generalizes Coud\`ene's one \cite[Chapter 9]{coudene2013book} on the two-torus, all results obtained in following sections have their counterparts in the two-torus case.
\end{remark}

\subsection{Smoothness and range of parameters}\label{sect:first_prop_f_K}

In order to ensure that the explicit construction introduced above makes sense, we need to ensure that the open sets $V_{\sigma}$ near each conical point do not overlap with one another, nor with themselves. This can be easily done by taking the parameter $\alpha_{\sigma}$ small enough. We give a simple  bound on their size by geometric considerations.

Let $Syst_{s.c}(S_g) = \inf \{ \D(\sigma_1,\sigma_2) \mid \sigma_1, \sigma_2 \in \Sigma \}$, where $\D$ is the distance for the flat metric on $S_g$ associated to the invariant measured-foliation of $\varphi$. 
Let $Syst(S_g) = \inf \{ l(\gamma) \mid \gamma \neq 0 \text{ in } \pi_1(S_g) \}$ be the smallest possible length of any non-trivial loop. Then define $\delta_{\Sigma} = \min( Syst_{s.c}(S_g), Syst(S_g))$. 

Actually, it is a general fact that $Syst_{s.c}(S_g) \leqslant Syst(S_g)$, and therefore $\delta_\Sigma = Syst_{s.c}(S_g)$. This can be seen from the fact that if $\gamma \neq 0$ in $\pi_1(S_g)$, we can assume that $\gamma$ is a straight line. Translating $\gamma$, we get that $\gamma$ is contained in a cylinder. Thus, considering the maximal cylinder containing $\gamma$, it must have a conical point in its boundary. In other words, there is a translation of $\gamma$ that is a saddle connection.

\needspace{2cm}
\begin{proposition}\label{prop:f_diffeo}
For all $\beta_{\sigma} \in ]-\lambda,0]$ and all $\alpha_{\sigma} < \delta_{\Sigma}/2$, $f$ is a homeomorphism on $S_g$ and is a $\mathcal{C}^1$ diffeomorphism on $S_g \smallsetminus \Sigma$.
\end{proposition}

\begin{proof}
Clearly, $f$ is continuous on $S_g$ and differentiable everywhere except on $\Sigma$. 
The differential on $S_g \smallsetminus \Sigma$ of $f$ is invertible, hence $f$ is a local homeomorphism on $S_g \smallsetminus \Sigma$ and hence $f(S_g \smallsetminus \Sigma)$ is open. In charts around points of $\Sigma$, one can see that $f$ is a local homeomorphism in a neighbourhood of $\Sigma$. 
Hence $f(S_g)$ is open. Since $S_g$ is compact, $f(S_g)$ is closed. Hence $f(S_g)=S_g$, because $S_g$ is connected. Therefore, $f$ is a surjective local homeomorphism, hence $f$ is a covering map. Since the pre-image of a point of $\Sigma$ by $f$ is itself, $f$ is injective.
\end{proof}

By refining the range where the $\beta_{\sigma}$ live, we can turn conical points into attractive fixed points.

\begin{proposition}
For $\beta_{\sigma} \in ]-\lambda,1-\lambda[$ and $\alpha_{\sigma} < \delta_{\Sigma}/2$, $\sigma \in \Sigma$ is an attractive fixed point for $f$. Let $U_{\sigma}$ be its basin of attraction. Then $U_{\sigma}$ is an open set.
\end{proposition}

\begin{proof}
Fix some $\sigma \in \Sigma$ and a $W_i$. Then, in charts centred at $\sigma$, the Jacobian matrix of $f$ is
\begin{align}\label{eq:jacobian_f}
(\mathrm{Jac} \, f) (x,y) \coloneqq \mathrm{Jac} \, \xi \circ f \circ \xi\vert_{W_i}^{-1} (x,y) = \begin{pmatrix}
\lambda + \beta_\sigma k_\sigma \left( \frac{|z|}{\alpha_\sigma} \right) + \frac{\beta_\sigma}{\alpha_\sigma} \frac{x^2}{|z|} k_{\sigma}' \left( \frac{|z|}{\alpha_\sigma} \right) & \frac{\beta_\sigma}{\alpha_\sigma} \frac{xy}{|z|} k_{\sigma}' \left( \frac{|z|}{\alpha_\sigma} \right) \\
0 & \lambda^{-1}
\end{pmatrix},
\end{align}
where $z = x + iy$. Evaluating at $(0,0)$, we get that $\mathrm{Jac} \, f (0,0) = \begin{pmatrix}
\lambda + \beta_\sigma & 0 \\ 0 & \lambda^{-1} \end{pmatrix}$. It is then a consequence of the Grobman--Hartman theorem that $U_\sigma$ contains a ball centred at $\sigma$. Finally, we have $U_{\sigma} = \bigcup\limits_{n \geqslant 0} f^{-n}(B(\sigma,\varepsilon))$, for some small enough $\varepsilon > 0$.
\end{proof}

Since basins of attraction $U_{\sigma}$ are disjoint open sets and $S_g$ is connected, these basins are not an open cover. Therefore the complement of the union of basins is not empty. Define $U_{\Sigma} \coloneqq \bigsqcup\limits_{\sigma \in \Sigma} U_{\sigma}$ and $K \coloneqq S_g \smallsetminus U_{\Sigma}$. These sets are clearly invariants by $f$.

\begin{proposition}\label{prop:existence_ball_in_basin}
If for some $\sigma \in \Sigma$, $\beta_{\sigma} \in ]-\lambda,1-\lambda[$ and $\alpha_{\sigma} < \delta_{\Sigma}/2$, then there exists a fixed hyperbolic point $p^{\sigma}_i$, $1 \leqslant i \leqslant 2 n_{\sigma}$, on each horizontal ray starting at $\sigma$. We number them by going counter-clockwise around $\sigma$. All these points are at the same distance $|p^{\sigma}|$ from $\sigma$. Moreover $B(\sigma, |p^{\sigma}|) \subset U_{\sigma}$.
\end{proposition}

\begin{proof}
Let $\sigma$, $\beta_{\sigma}$ and $\alpha_{\sigma}$ be as in the proposition. Let $\gamma : [0,\alpha_{\sigma}] \to S_g$ be a unit speed parametrization of a horizontal ray such that $\gamma(0)=\sigma$. Hence, in charts, $f(\gamma(t)) = ((\lambda + \beta_{\sigma} k(t/\alpha_{\sigma}))t,0)$. Let $h:[0,\alpha_{\sigma}] \to \mathbbm{R}$ be the function $h(t)= (\lambda + \beta_{\sigma} k(t/\alpha_{\sigma}))t$. Then $h(0)=0$, $h(\alpha_{\sigma}) = \lambda \alpha_{\sigma} > \alpha_{\sigma}$, and $h'(0)=\lambda + \beta_{\sigma} \in \, ]0,1[$. Hence $h$ has a fixed point in $]0,\alpha_{\sigma}[$. Call $t_0$ the smallest fixed point. This value doesn't depend on which horizontal ray starting from $\sigma$ we consider. The point $p=\gamma(t_0)$ is fixed by $f$ and is hyperbolic: in the charts centred at $\sigma$, the Jacobian matrix of $f$ at $p$ is \[ (\mathrm{Jac} \, f)(p) = \begin{pmatrix}
1 + \frac{\beta_{\sigma}}{\alpha_{\sigma}} t_0  k_{\sigma}' \left( \frac{t_0}{\alpha_{\sigma}} \right) & 0 \\ 0 & \lambda^{-1}
\end{pmatrix}, \]
where we used that $\lambda + \beta_\sigma k_\sigma \left( \frac{t_0}{\alpha_\sigma} \right) =1$ because $t_0$ is a fixed point of $h$. Now, since $\frac{\beta_{\sigma}}{\alpha_{\sigma}} t_0  k_{\sigma}' \left( \frac{t_0}{\alpha_{\sigma}} \right) > 0 $, $p$ is a hyperbolic fixed point.

By definition of $t_0$, we have $\gamma([0,t_0[) \subset U_{\sigma}$. Let $z \in B(\sigma,t_0)$. In the appropriate leaf of the branched-cover over $\sigma$, we have $z=(x,y)$ in coordinates. Hence,
\begin{align*}
\D(f(z),\sigma)^2 &\leqslant (\lambda + \beta_{\sigma} k(d(z,\sigma)/\alpha_{\sigma}))^2 x^2 + \lambda^{-2} y^2 \\
&< (\lambda + \beta_{\sigma} k(t_0/\alpha_{\sigma}))^2 x^2 + \lambda^{-2} y^2 \leqslant x^2 + \lambda^{-2} y^2 \leqslant \D(z,\sigma)^2.
\end{align*}
Hence, the function $z \mapsto \D(f(z),\sigma)/\D(z,\sigma)$ is continuous and strictly bounded from above by $1$ on the compact annulus $\{ z \in S_g \mid \varepsilon \leqslant \D(z,\sigma) \leqslant t_0-\varepsilon \}$. Therefore every orbit of point from the ball $B(\sigma,t_0-\varepsilon)$ ends up entering the ball $B(\sigma,\varepsilon)$. Hence the claim.
\end{proof}

\subsection{Invariant sets}

Here we investigate the topological aspects of the invariant set $K$. In particular we prove that it can be written as the union of the closure of some stable leaves of the hyperbolic fixed points $p_{i}^{\sigma}$ and that it is a hyperbolic set.

We start by proving that the set $U_{\Sigma}$ is dense in $S_g$, or equivalently that $K$ is of empty interior. In order to do this we need the following lemma which is obtained by simply computing the differential of $f$.

\begin{lemma}
Define $(q^{\sigma}_i)_i$ as the $2 n_{\sigma}$ points at distance $|p^{\sigma}|$ from $\sigma$ on the \emph{vertical} rays starting from $\sigma$. Then for all $x \in S_g \smallsetminus \bigsqcup\limits_{\sigma \in \Sigma} ( B(\sigma,|p^{\sigma}|) \cup \{ q^{\sigma}_i \mid 1 \leqslant i \leqslant 2 n_{\sigma} \})$, $f$ is a strict dilation in the horizontal direction.
\end{lemma}

\begin{proof}
Fix some $\sigma \in \Sigma$. We want to prove that the upper-left coefficient in \eqref{eq:jacobian_f} is strictly larger than $1$ on the claimed domain. Call $a(x,y)$ this coefficient. For $t > |p^{\sigma}|$, we have that $(\lambda + \beta_\sigma k_\sigma(t/\alpha_\sigma))t > t$. Then, for any $\theta \in \mathbbm{R}$, $ a(t \cos(\theta), t \sin(\theta)) = \lambda + \beta_\sigma k_\sigma(t/\alpha_\sigma) + \frac{\beta_\sigma}{\alpha_\sigma}t(\cos \theta)^2 k_{\sigma}'(t/\alpha_\sigma) > 1$.

On the other hand, $a(|p^{\sigma}| \cos(\theta), |p^{\sigma}| \sin(\theta)) = 1 + \frac{\beta_\sigma}{\alpha_\sigma}|p^{\sigma}|(\cos \theta)^2 k_{\sigma}'(|p^{\sigma}|/\alpha_\sigma) $ is equal to $1$ if and only if $\cos \theta = 0$, or in other words, $(|p^{\sigma}| \cos(\theta), |p^{\sigma}| \sin(\theta)) = q_i^{\sigma}$ for some $1 \leqslant i \leqslant 2 n_\sigma$.
\end{proof}

\begin{proposition}
For all $z \in K$ and all $\varepsilon >0$, every horizontal segment of length $\varepsilon$ containing $z$ in its interior crosses $U_{\Sigma}$. Hence $U_{\Sigma}$ is dense and $K$ has empty interior.
\end{proposition}

\begin{proof}
By contradiction, let $\gamma : [-\varepsilon,\varepsilon] \to S_g$ be a horizontal segment parametrized by arc length, containing some $z \in K$ and such that $\gamma([-\varepsilon,\varepsilon]) \cap U_{\Sigma} = \emptyset$. Without loss of generality, we can assume that $\gamma(0)=z$. Since $U_{\Sigma}$ is invariant by $f$, we see that the existence of some $-\varepsilon \leqslant t \leqslant \varepsilon$ such that $f^n(\gamma(t)) \in U_{\Sigma}$ is impossible. Hence $f^n(\gamma([-\varepsilon,\varepsilon])) \cap U_{\Sigma} = \emptyset$. By construction of $f$, the set $f^n(\gamma([-\varepsilon,\varepsilon]))$ is a horizontal segment, containing $f^n(z)$ in its interior and of length $l_n$. Since $f$ is a strict dilation in the vertical direction on the compact set $K$, there exists $l_* >1$ such that $l_n \geqslant l_*^n$.

Let $\delta = \inf \{ |p^{\sigma}| \mid \sigma \in \Sigma \}$ and since $K$ is compact and invariant by $f$ let $z' \in K$ be a subsequential limit of $(f^n(z))_n$. Let $n_k$ be an increasing sequence of integers such that $f^{n_k}(z)$ converges to $z'$ as $n_k$ goes to infinity. We know -- see \cite[corollary 14.15]{farb2011primer} -- that the vertical leaf containing $z'$ is at least infinite in one direction and is dense in $S_g$. In particular, some sufficiently long section of this leaf, containing $z'$, is $\delta/4$-dense in $S_g$. Hence, for large enough $n_k$, the curve $f^n(\gamma([-\varepsilon,\varepsilon]))$ is sufficiently long and sufficiently close to the horizontal leaf containing $z'$ to be $\delta/2$-dense in $S_g$. In particular, there exists $-\varepsilon < t < \varepsilon$ such that $\D(f^n(\gamma(t)),\sigma) < \delta$ for some $\sigma \in \Sigma$. This contradicts the fact that $B(\sigma,|p^{\sigma}|) \subset U_{\Sigma}$.
\end{proof}

Recall definitions of strong stable and strong unstable leaves of $x \in S_g$ with respect to $f$
\begin{align*}
W^{ss}(x) &= \{ y \in S_g \mid \D(f^{n}(x),f^{n}(y)) \to 0 \text{ as } n \to +\infty \}, \\
W^{su}(x) &= \{ y \in S_g \mid \D(f^{-n}(x),f^{-n}(y)) \to 0 \text{ as } n \to +\infty \}.
\end{align*}
If $x$ is a fixed point of $f$, then these sets are invariant by $f$. 

Here, these leaves at hyperbolic fixed points $p^{\sigma}_{i}$ enable to describe precisely the set $K$. We start by showing that the stable leaves can be seen as the \emph{horizontally accessible border} of $U_{\Sigma}$\footnote{The accessible border of an open set $U \subset X$ is the set of points of $\partial U$ that are the endpoints of a curve $\gamma : [0,1] \to X$ contained in $U$ for $t \in [0,1)$. The horizontally accessible border is obtained by assuming that the curves $\gamma$ are horizontal.} -- and are obviously contained in $K$. On the other hand, unstable leaves are dense.

\begin{proposition}\label{prop:stable_leaves} The stable and unstable leaves of the fixed point $p^{\sigma}_i$ satisfies the following assertions.
\begin{enumerate}[label=(\roman*), wide]
\item If $x \in U_{\sigma}$ and $\gamma : [0,1] \to S_g$ is a horizontal\footnote{Using the flow $h_t$ introduced in Section~\ref{subsect:h_t}, one can bootstrap from Proposition~\ref{prop:stable_leaves}(i) and remove the assumption that $\gamma$ is horizontal.} curve such that $\gamma(0)=x$, $\gamma([0,1[) \subset U_{\sigma}$ and $\gamma(1) \notin U_{\sigma}$ then $\gamma(1)$ belongs to $\bigsqcup\limits_{1 \leqslant i \leqslant 2n_{\sigma}} W^{ss}(p_{i}^{\sigma})$.
\item For all $\sigma \in \Sigma$ and all $1 \leqslant i \leqslant n_{\sigma}$, the unstable leaf $W^{su}(p^{\sigma}_i)$ contains a full semi-infinite horizontal leaf. Hence $W^{su}(p^{\sigma}_i)$ is dense in $S_g$.
\end{enumerate}
\end{proposition}

\begin{proof} We begin with the first point. Let $\delta > 0$ be the length of the smallest side in the (finite) collection of rectangle neighbourhoods of points $p^{\sigma}_i$ given by the Grobman--Hartman theorem. 
For $n$ large enough, we find that $f^n(x)$ is $\delta/2$-close to some $\sigma \in \Sigma$. Once close to $\sigma$ by going horizontally along $f^n \circ \gamma$, the first time $f^n \circ \gamma$ intersects $K$ is at some point contained in one of the rectangle neighbourhood of some $p^{\sigma}_i$. Therefore this intersection point belongs to $\bigsqcup\limits_{1 \leqslant i \leqslant 2n_{\sigma}} W^{ss}(p_{i}^{\sigma})$ and is attained at $f^n(\gamma(1))$. The result then follows from the invariance by $f$ of the stable leaves.

We now prove the second point. Let $\gamma : [0, +\infty [ \to S_g$ be a unit speed parametrization of the horizontal ray starting at $\sigma \in \Sigma$ and containing $p \coloneqq p^{\sigma}_i$. In particular, $\gamma(0)=\sigma$ and $\gamma(|p^{\sigma}|)=p$. 

By contradiction, assume there exists $t\geqslant |p^{\sigma}|$ such that $\gamma(t) \notin W^{su}(p)$. Let $t_0 = \inf \{ t \geqslant |p^{\sigma}| \mid \gamma(t) \notin W^{su}(p) \}$.

We now show that $t_0 > |p^{\sigma}|$. Let $h : t \mapsto (\lambda + \beta_{\sigma} k (t/\alpha_{\sigma}))t$. By construction of $f$, we have the relation $f(\gamma(t))=\gamma(h(t))$ for every $t \in [0, \alpha_{\sigma}[$, and hence $f^n(\gamma(t))=\gamma(h^n(t))$ for all $n \in \mathbbm{Z}$. Now $(h^{-1})'(|p^{\sigma}|) < 1$, so for $t$ close to $|p^{\sigma}|$, $f^{-n}(\gamma(t)) \to p$ as $n$ goes to infinity. Therefore $t_0 > |p^{\sigma}|$.

We now prove that $\gamma(t_0)$ is a fixed point of $f$. We know that $f(\gamma(]|p^{\sigma}|,t_0[) = \gamma(]|p^{\sigma}|,s[)$ for some $s$. But $f(\gamma(]|p^{\sigma}|,t_0[) \subset W^{su}(p)$. Hence $s \leqslant t_0$.

By contradiction, assume there exists $\varepsilon > 0$ such that $s + \varepsilon < t_0$. So $\gamma([|p^{\sigma}|, s + \varepsilon[) \subset W^{su}(p)$, and so $f^{-1} \circ \gamma ([|p^{\sigma}|,s+\varepsilon [) \subset W^{su}(p)$. However, $f^{-1} \circ \gamma ([|p^{\sigma}|,s+\varepsilon [) \subset W^{su}(p) = \gamma([|p^{\sigma}|,t_0 + \delta_{\varepsilon}[)$ for some $\delta_{\varepsilon}>0$ since $f$ is strictly preserving horizontal orientation. This contradicts the definition of $t_0$. Therefore $s=t_0$ and $\gamma(t_0)$ is fixed by $f$.

The point $\gamma(t_0)$ cannot be in $\Sigma$ nor be a $p^{\sigma}_i$, otherwise $\gamma$ would connect two conical points, which is impossible.  By computing the differential of $f$ at $\gamma(t_0)$, we see that $\gamma(t_0)$ is a hyperbolic fixed point of $f$ with a horizontal unstable leaf. Therefore there exist points whose iterates by $f^{-1}$ converge to $p$ and to $\gamma(t_0) \neq p$. 
\end{proof}

These properties of stable and unstable leaves yield to the fact that the set $K$ can be written as a finite union of closure of stable leaves. In fact, we have the following slightly stronger result.

\begin{proposition}\label{prop:K_as_Finite_Union}
The compact set $K$ can be written as a finite union of closed invariant sets as follow $K = \bigcup\limits_{\sigma \in \Sigma} \bigcup\limits_{i=1}^{n_{\sigma}} \overline{W^{ss}(p^{\sigma}_i) \cap W^{su}(p^{\sigma}_i) }$ .
\end{proposition}

\begin{proof}
Let $x \in K$ and $\varepsilon >0$. Let $y \in U_{\Sigma}$ be in the same horizontal leaf as $x$ and obtained by \emph{going leftward} by a distance less than $\varepsilon$. Since $U_{\Sigma} = \bigsqcup\limits_{\sigma} U_{\sigma}$, there exists $\sigma \in \Sigma$ such that $y \in U_{\sigma}$. From the Grobman--Hartman theorem, for each $1 \leqslant i \leqslant 2 n_{\sigma}$ there exists a neighbourhood of $p^{\sigma}_i$ on which the dynamic of $f$ is the same as the one of the differential of $f$. Without loss of generality, we assume that these neighbourhoods are rectangles with vertical and horizontal sides and with centers the $p^{\sigma}_i$'s. Up to replacing these rectangles by smaller ones, let $\delta_{\sigma}$ be a common vertical size for these rectangles.

For $n\geqslant 0$ large enough, the point $y$ lies in $B(\sigma, \delta_{\sigma}/4)$. By construction and by the first point of Proposition~\ref{prop:stable_leaves}, we know that by going \emph{rightward} from $y$ we cross some $W^{ss}(p^{\sigma}_i)$, for some $1 \leqslant i \leqslant 2 n_{\sigma}$. Therefore, by going \emph{rightward} from $f^{-n}(y)$ we cross the rectangle of linearisation associated with $p^{\sigma}_i$, and hence the stable leaf $W^{ss}(p^{\sigma}_i)$ at some point $y^r$.

Let $\delta$ be the modulus of absolute continuity of $f^{-n}$ associated with $\varepsilon$. By density of the unstable leaf of $p^{\sigma}_i$, we can chose a point $z$ such that $\D(f^n(z),f^n(y))< \min(\delta, \delta_{\sigma}/4)$ so that by going \emph{rightward} from $f^n(z)$ we cross $W^{ss}(p^{\sigma}_i)$ at some point $z^r$, at distance less than $\delta$ from $y^r$. Finally, the point $f^{-n}(z^r) \in W^{ss}(p^{\sigma}_i) \cap W^{su}(p^{\sigma}_i)$ is at distance less than $3\varepsilon$ from $x$.
\end{proof}

Finally, we explicit stable and unstable foliations such that the set $K$ is hyperbolic with respect to $f$. To do this, we compute a vector field that is uniformly contracted by the differential of $f$.

\begin{theorem}\label{thm:K_is_hyperbolic}
The set $K$ is hyperbolic. The invariant distributions are $E^u(x) = \mathbbm{R}e_h$ and $E^s(x) = \mathbbm{R}v^s(x)$, $x \in K$, where \begin{align}\label{eq:definition_v^s}
v^s(x) \coloneqq e_v - \sum\limits_{i \geqslant 0} \lambda^{-i} b(f^i(x)) \prod\limits_{j=0}^{i} \frac{1}{a(f^j(x))} \, e_h,
\end{align}
with $a(x) \coloneqq \langle \D_x f \cdot e_h, e_h \rangle$ and $b(x) \coloneqq \langle \D_x f \cdot e_v, e_h \rangle$. In particular, $v^s$ satisfies $\D f \, v^s = \lambda^{-1} v^s \circ f$ on $K$.
\end{theorem}

\begin{proof}
We will explicit the stable and the unstable directions of the splitting of the tangent space. 
Write the differential of $f$ at $x \in S_g \smallsetminus \Sigma$ in the basis $(e_h,e_v)$
\begin{align*}
\D_x f = \begin{pmatrix} a(x) & b(x) \\ 0 & \lambda^{-1} \end{pmatrix}.
\end{align*}
Therefore, for every positive integer $n$, we have the following,
\begin{align*}
\D_x (f^n) &= \D_{f^{n-1}(x)} f \cdots \D_{f(x)} f \, \D_x f ,\\
&= \begin{pmatrix} a(f^{n-1}(x)) & b(f^{n-1}(x)) \\ 0 & \lambda^{-1} \end{pmatrix} \cdots \begin{pmatrix} a(f(x)) & b(f(x)) \\ 0 & \lambda^{-1} \end{pmatrix}  \begin{pmatrix} a(x) & b(x) \\ 0 & \lambda^{-1} \end{pmatrix}, \\
&= \begin{pmatrix} A_n(x) & B_n(x) \\ 0 & \lambda^{-n} \end{pmatrix}.
\end{align*}
We have that $A_n(x) = \prod\limits_{i=0}^{n-1} a(f^{i}(x))$. We compute $B_n$ explicitly. This sequence satisfies a recursive formula, which can be solved
\begin{align*}
B_{n+1} - a(f^{n}) B_n &= \lambda^{-n} b(f^{n}), \\
B_{n+1}/A_{n+1} - B_n/A_n &= \lambda^{-n} b(f^{n})/A_{n+1}, \\
B_n/A_n &= \sum\limits_{i= 0}^{n-1} \lambda^{-i} b(f^{i})/A_{i+1} .
\end{align*}
Finally we get \[ \frac{B_n}{A_n}(x) = \sum\limits_{i=0}^{n-1} \lambda^{-i} b(f^i(x)) \prod\limits_{j=0}^{i}\frac{1}{a(f^{j}(x))}. \]

We can now explicit the eigenvectors of $\D_x (f^n)$. The obvious one, associated with the eigenvalue $A_n(x)$, is $e_h$. The other one is \[ v_n(x) = \begin{pmatrix} -B_n(x)/A_n(x) \\ 1- \lambda^{-n} /A_n(x) \end{pmatrix} = \begin{pmatrix} -\sum\limits_{i=0}^{n-1} \lambda^{-i} b(f^i(x)) \prod\limits_{j=0}^{i}\frac{1}{a(f^{j}(x))} \\ 1 - \prod\limits_{i=0}^{n-1}\frac{1}{\lambda a(f^i(x))} \end{pmatrix} . \]
We now study the convergence of the $v_n$'s as $n$ goes to infinity. First, since $a > 1$, $b$ are continuous functions over the compact set $K$, there exist constants $a^*$ and $C$ such that $a > a^* > 1$ and $|b| < C$. Therefore, the second coordinate converges to $1$ as $n$ goes to infinity. For the first coordinate, we have the uniform bound over $K$
\begin{align*}
\sum\limits_{i=0}^{n-1} \left| \lambda^{-i} b(f^i(x)) \prod\limits_{j=0}^{i}\frac{1}{a(f^{j}(x))} \right| &\leqslant C \sum\limits_{i=0}^{n-1} (\lambda a^*)^{-i} \leqslant C \frac{\lambda a^*}{\lambda a^* - 1}.
\end{align*}
Hence, the series of continuous functions converges uniformly over $K$ to a continuous function. Call $v^s$ the limit of $v_n$ as $n$ goes to infinity.

A short computation shows that for all $x$ in $K$, $v^s$ satisfies $\D_x f \cdot v^s(x) = \lambda^{-1} v^s(f(x))$. Finally, we get the following splitting of the tangent space at each $x$ in $K$, $T_x S_g = \mathbbm{R}v^s(x) \oplus \mathbbm{R}e_h$, so that $K$ is a hyperbolic set.
\end{proof}

\section{Smoothness of the stable foliation and renormalized flow}\label{sect:technicalities}

In this section we prove that $v^s$ can be extended to the whole set $S_g \smallsetminus \Sigma$ of regular points, such that the extension is still uniformly contracted by the action of $f$ (\ref{eq:commut_diff}) and so that it is Lipschitz continuous. Under further assumption on the smoothness of $f$, we prove that $v^s$ is $\mathcal{C}^1$. Furthermore, in view of the next section, we prove that $v^s$ depends continuously on the parameter $\beta$ -- occuring in the construction of $f$. This regularity property will be crucial in Section~\ref{sect:f_is_mixing}. Since $v^s$ is Lipschitz continuous, it can be integrated into a continuous flow $h_t$ which enjoys the commutation relation (\ref{eq:commut_flow}) with $f$ -- in other words, $f$ renormalizes $h_t$. From the properties of $h_t$, we show that the set $K$ is connected, transverse to any vertical leaf, and that $f$ is transitive with respect to the trace topology on $K$.

\subsection{Construction of a useful open cover of $S_g$} 

In order to proceed, we first need to construct an open cover of $S_g \smallsetminus \Sigma$ such that $f$ satisfies some nice estimates on elements of this cover. This is done in the following proposition.

\begin{proposition}\label{prop:open_cover}
For all $\varepsilon>0$ small enough, there exist $\eta > 0 $, $\delta >0$, and an open cover $S_g = A_{\eta} \cup \bigsqcup\limits_{\sigma \in \Sigma} B_{\sigma,\delta}$ such that $a> 1 + \eta$ on $A_{\eta}$ and $\D(f(x), \sigma) < (1 - \delta) \D(x,\sigma)$ on $B_{\sigma,\delta} \smallsetminus \{ \sigma \}$.
\end{proposition}

\begin{proof}
By continuity of $f$, there exists an $\varepsilon > 0$ such that \[ \{ x \in V_{\sigma} \mid \D(f(x),\sigma) < \D(x,\sigma) \} \supset B(\sigma,|p^{\sigma}|)\cup \bigcup\limits_{i=1}^{2n_{\sigma}} B(q_i^{\sigma},\varepsilon) \eqqcolon B_{\sigma}^{\varepsilon}, \] for all $\sigma$, where $V_{\Sigma} = \bigsqcup\limits_{\sigma \in \Sigma} V_{\sigma}$ is the open neighbourhood of $\Sigma$ on which $f \not\equiv \varphi$.

Since $S_g \smallsetminus \bigsqcup\limits_{\sigma \in \Sigma} B_{\sigma}^{\varepsilon}$ is compact and $a >1$ on it, there exists $\eta >0$ such that $a > 1 + 2\eta$ on this compact set. Call $A_{\eta} = \{x \in S_g \mid a > 1 + \eta \}$. By construction, $S_g = A_{\eta} \cup \bigcup\limits_{\sigma \in \Sigma} B_{\sigma}^{\varepsilon}$.

Since all $B_{\sigma}^{\varepsilon}$ are open sets, radial and centred on $\sigma$, we have $B_{\sigma}^{\varepsilon} = \bigcup\limits_{n \geqslant 1} \left(1 - \tfrac{1}{n} \right) B_{\sigma}^{\varepsilon}$. Now, by compactness of $S_g$, there exists $n_0$ such that: \[ S_g = A_{\eta} \cup \bigcup\limits_{\sigma \in \Sigma} \left(1 - \tfrac{1}{n_0} \right) B_{\sigma}^{\varepsilon}. \]

On a small open neighbourhood $W_{\sigma}$ of $\sigma$, by construction of $f$ we have that $\D(f(x),\sigma)/ \D(x,\sigma) < C < 1$. Now, on the compact set $\overline{(1-\tfrac{1}{2n_0}) B_{\sigma}^{\varepsilon}} \smallsetminus W_{\sigma}$, the continuous function $\D(f(x),\sigma)/ \D(x,\sigma)$ is positive and strictly bounded from above by $1$. On the other hand, up to shrinking $W_{\sigma}$, the function $\D(f(x),\sigma)/ \D(x,\sigma)$ is bounded on $W_{\sigma} \smallsetminus \{ \sigma \}$ by $\max( \lambda^{-1}, \lambda + \beta_{\sigma} + \tilde{\delta}) < 1$, for some small $\tilde{\delta}>0$. Hence, there exists $\delta > 0$, independent of $\sigma$, such that for all $x$ in $(1- \tfrac{1}{n_0})B_{\sigma}^{\varepsilon} \smallsetminus \{ \sigma \}$, $\D(f(x),\sigma) < (1-\delta) \D(x,\sigma)$. We then call $B_{\sigma,\delta} = (1- \tfrac{1}{n_0})B_{\sigma}^{\varepsilon}$.
\end{proof}

\subsection{Lipschitz extension of $v^s$ to $S_g \smallsetminus \Sigma$}

Here we prove that the infinite sum in the definition of the vector field $v^s$ on $K$ does converge on all $S_g \smallsetminus \Sigma$. This way we can define $v^s$ on $S_g \smallsetminus \Sigma$. Furthermore, we prove that this extended vector field is Lipschitz continuous.

We proceed in two steps. First we show that $v^s$ is bounded and continuous on $S_g \smallsetminus \Sigma$. To do this, we need a lemma which follows directly from computation of $\mathrm{d}f$.

\begin{lemma}\label{lemma:min_a}
On each basin $U_{\sigma}$, the partial derivative $a = \langle \D f (e_h), e_h \rangle$ of $f$ is bounded from below by $\lambda + \beta_{\sigma} $.

The partial derivative $b = \langle \D f (e_v), e_h \rangle$ of $f$ is locally Lipschitz in some neighbourhood of $\Sigma$. Furthermore, by continuity we can set $b(\sigma)=0$ for each $\sigma \in \Sigma$.
\end{lemma}

\begin{theorem}\label{thm:vs_bounded_continuous}
If $\beta_{\sigma} \in \, ]-\lambda + \lambda^{-2}, -\lambda +1[$ for all $\sigma$ in $\Sigma$, then the vector field $v^s$ is bounded and continuous on $S_g \smallsetminus \Sigma$. Furthermore, by construction, the formula $\D f(v^s) = \lambda^{-1} v^s \circ f$ holds on $S_g \smallsetminus \Sigma$.
\end{theorem}

\begin{proof}
Call $s_i = \lambda^{-i} \, b \circ f^i \, \prod\limits_{j=1}^{i} \tfrac{1}{a \circ f^j}$. Let $V$ be a neighbourhood of some $\sigma$ such that $b$ is Lipschitz on it and $f$ contracts by a factor $\max(\lambda^{-1}, \lambda + \beta_{\sigma} + \delta_{\sigma}) < 1$. Without loss of generality, we assume that $V$ is a ball centred at $\sigma$ of radius $\varepsilon$ and that $f(V) \subset V$. Since $U_{\sigma} = \bigcup\limits_{N \geqslant 0} f^{-N}V$, for all $x \in U_{\sigma}$ there exist some $N = N(x)$ and an integer $n_V$ which only depends on $V$, such that for all $n \geqslant N$, $f^n(x) \in V$, at most $n_V$ points of the orbits fall into $B_{\sigma,\delta} \smallsetminus V$ and the rest lives in $A_{\eta}$.

Let $x \in U_{\sigma}$, $x \neq \sigma$. Since $U_{\sigma} = \bigcup\limits_{n \geqslant 0} f^{-n}V$, let $N$ be the smallest integer such that $f^N(x) \in V$.
We distinguish three cases :
\begin{enumerate}[label=$\bullet$, wide]
\item $i \leqslant N - n_V$. Therefore $|s_i(x)| \leqslant \lambda^{-i} \left( \frac{1}{1+\eta} \right)^{i+1} \sup|b|$.
\item $N-n_V < i \leqslant N$. Hence $|s_i(x)| \leqslant \lambda^{-i} \left( \frac{1}{1+\eta} \right)^{N-n_V} \left( \frac{1}{\lambda + \beta} \right)^{i-(N-n_V)} \sup|b|$.
\item $i = j + N > N$. We get $|s_i(x)| \leqslant \lambda^{-(j+N)} \left( \frac{1}{\lambda + \beta } \right)^{j+N} \Lip(b) \varepsilon \max(\lambda^{-1}, \lambda + \beta_{\sigma} + \delta_{\sigma}) ^j$.
\end{enumerate}
Therefore, if $\lambda^{-2} < \lambda + \beta_{\sigma}$, then \[ \sum\limits_{i \geqslant 0} |s_i(x)| \leqslant \sup|b|  \frac{\lambda(1+\eta)}{ \lambda(1+ \eta) -1} \left( 1 + \sum\limits_{i=0}^{n_V} \left( \frac{1}{\lambda + \beta} \right)^i \right) + \frac{\Lip(b) \varepsilon}{ 1 - \frac{\max(\lambda^{-1}, \lambda + \beta_{\sigma} + \delta_{\sigma})}{\lambda(\lambda + \beta_{\sigma})}}, \]
which is uniform in $x$ on $U_{\sigma}$. Hence, the convergence is uniform on the compact subsets of $U_{\sigma}\smallsetminus \{ \sigma \}$ and $\sum s_i$ is continuous on $U_{\sigma}\smallsetminus \{ \sigma \}$, for all $\sigma \in \Sigma$.

We now show that this function defined on $U_{\Sigma} = \sqcup U_{\sigma}$ can be extended by continuity on $K$. Call $u(x) = e_v - \sum\limits_{i \geqslant 0} s_i(x) e_h$ the vector based at $x \in S_g \smallsetminus \Sigma$. 

Let $x \in K$ and, by density of $U_\Sigma$ in $S_g$, $(x_n)_n \in U_{\Sigma}^{\mathbbm{N}}$ such that $x_n \to x$ as $n$ goes to infinity. Since $(u(x_n))_n$ is bounded, up to extracting, the sequence converges to some $u_0$. Furthermore, by a diagonal argument and up to extracting, $u(f^k(x_n)) \to u_k$ for all $k \in \mathbbm{Z}$ as $n$ goes to infinity. Now, by construction of $u$, $\D f(u) = \lambda^{-1} u \circ f$. Hence, by continuity of $f$ and $\D f$, $\D_x (f^k)(u_0) = \lambda^{-k} u_k$. We now show that $u_0 = v^s(x)$. By hyperbolicity of $K$, there exist real numbers $x_s$, $x_u$ such that $u_0 = x_s v^s(x) + x_u e_h$. Therefore, by hyperbolicity of $K$, \begin{align*}
|x_u| &= ||x_u e_h||, \\
&= || \D_{f^k(x)} f^{-k} \D_x f^k x_u e_h ||, \\
&\leqslant C (\tfrac{1}{a_*})^k || \D_x f^k x_u e_h ||, \\
&= C (\tfrac{1}{a_*})^k || \D_x f^k (u_0 - x_s v^s(x))||, \\
&\leqslant C (\tfrac{1}{a_*})^k \lambda^{-k} (\sup||u|| + x_s \sup||v^s|| ),
\end{align*}
which goes to zero as $k$ goes to infinity. Hence $u_0 = x_s v^s(x)$. Now both $u_0$ and $v^s(x)$ have the same non-zero coordinate along $e_v$ in the base $(e_h,e_v)$. Hence $u_0 = v^s(x)$. Finally, $u$ extends continuously on $K$ by $v^s$. We call $v^s$ this vector field on $S_g \smallsetminus \Sigma$.
\end{proof}

We can now present the proof of the Lipschitz continuity of $v^s$ on $S_g \smallsetminus \Sigma$. To this end, we need a few more estimates on the differential of $f$ and on its coefficients.

\begin{lemma}
For all $x \in S_g \smallsetminus \Sigma$, the following estimate holds \[ \frac{||\diff_x f^n||}{ A_n(x)} \leqslant 2\max\left(1, \frac{|B_n|(x) + \lambda^{-n}}{A_n(x)} \right). \] 
In particular, $||\diff f^n|| / A_n$ is bounded on $\bigcup\limits_{i=0}^{n} f^{-i} A_{\eta}$. Furthermore, the bound $B$ can be chosen independently of $n$.
\end{lemma}

\begin{proof}
By a direct computation, for $(u,v) \coloneqq u e_h + v e_v$ \begin{align*}
|| \diff_x f^n (u,v)||^2 &= (A_n(x) u + B_n(x) v)^2 + (\lambda^{-1} v)^2, \\
&\leqslant 4 A_n(x)^2 u^2 + (4 B_n(x)^2 + \lambda^{-2n})v^2, \\
&\leqslant 4 \max (A_n(x)^2, B_n(x)^2 + \lambda^{-2n}) ||(u,v)||^2.
\end{align*}
For $x \in \bigcup\limits_{i=0}^{n} f^{-i} A_{\eta}$, we know that $\lambda^{-k}/A_k(x) < (\lambda ( 1+ \eta))^{-k}$ and that $-B_n/A_n$ is the partial sum of $\sum s_i$, hence uniformly bounded.
\end{proof}

The following lemma is a direct consequence of the Lipschitz continuity of $k'$ intervening in the construction of $f$, and of Lemma~\ref{lemma:min_a}.

\begin{lemma}
The functions $a$ and $\tfrac{1}{a}$ are Lipschitz continuous on $S_g$.
\end{lemma}

\begin{theorem}\label{thm:vs_lipschitz}
If $\beta_{\sigma} \in ] - \lambda + \lambda^{-2}, -\lambda + 1 [$ for all $\sigma$ in $\Sigma$, then the vector field $v^s$ is Lipschitz continuous on $S_g \smallsetminus \Sigma$.
\end{theorem}

\begin{proof}
Since all of the partial sums of $\sum s_i$ are Lipschitz continuous, we give summable estimates of local Lipschitz constants. Let $x \in U_{\sigma}$. Let $V$, $N=N(x)$ and $n_V$ be as in the proof of Theorem~\ref{thm:vs_bounded_continuous}. Therefore $U_{\sigma}= \bigcup_{n \geqslant 0} f^{-n}V$. We use the notation $\Lip_x(g)$ to indicate the local Lipschitz constant of a function $g$ in at least one neighbourhood of $x$.

Let $\varepsilon >0$. On a small enough neighbourhood of $x$, we have that $\Lip_x(f^j) \leqslant (1+ \varepsilon )|| \diff_x f^j ||$  and $ \sup \tfrac{1}{A_j} \leqslant (1+\varepsilon) \tfrac{1}{A_j(x)} $ for all $j \leqslant i$. We distinguish the three following cases:
\begin{enumerate}[label=$\bullet$, wide]
\item $i \leqslant N - n_V$. We have directly that,
\begin{align*}
\Lip_x(s_i) &\leqslant \lambda^{-i} \left( \Lip(b) \Lip(f^i) \sup\tfrac{1}{A_i} + \sup(b \circ f^i) \Lip\tfrac{1}{a} \sum\limits_{j=0}^{i} \Lip(f^j) \sup \tfrac{1}{A_{j-1}} \sup \tfrac{A_j}{A_i} \right), \\
&\leqslant \lambda^{-i} B (1+\varepsilon)^2 \left( \Lip(b) + \sup|b| \Lip \tfrac{1}{a} \sup(a) \sum\limits_{j=0}^i \left( \frac{1}{1+\eta} \right)^j \right), \\
&\leqslant C_{\indep i, N, x} \, \lambda^{-i},
\end{align*}
where $C_{\indep i, N, x}$ stands for a constant independent of $i$, $N$ and $x$.
\item $N- n_V \leqslant i < N$. Up to multiplying some part of the above estimate by $( \tfrac{1}{\lambda + \beta_{\sigma}} )^{n_V}$, we have: \begin{align*}
\Lip_x(s_i) &\leqslant C_{\indep i, N, x} \, \lambda^{-i}.
\end{align*}
\item $i=l+N \geqslant N$. In this case, the following estimates hold:
\begin{align*}
\Lip_x(b \circ f^{l+N}) \sup\frac{1}{A_{l+N}} &\leqslant \Lip(b) \Lip(f^N) \sup \frac{1}{A_N} \Lip_{f^N(x)}(f^l) \sup \frac{A_N}{A_{l+N}}, \\
&\leqslant \Lip(b) (1+\varepsilon)^2 \frac{||\diff_x f^N||}{A_N(x)} \max(\lambda^{-1}, \lambda +\beta_{\sigma} + \delta_{\sigma})^l \left( \frac{1}{\lambda + \beta_{\sigma}} \right)^l, \\
&\leqslant C_{\indep x,i,N} \, \max\left( \frac{\lambda^{-1}}{\lambda + \beta_{\sigma}}, 1 + \frac{\delta_{\sigma}}{\lambda + \beta_{\sigma}} \right)^l .
\end{align*}
\begin{align*}
\sup(b\circ f^{l+N}) \Lip_x \frac{1}{A_{l+N}} &\leqslant \varepsilon \max(\lambda^{-1}, \lambda +\beta_{\sigma} + \delta_{\sigma})^l \Lip \tfrac{1}{a} \left( \sum\limits_{j=0}^{N-1} Lip(f^j) \sup \frac{1}{A_{j-1}} \sup\frac{A_j}{A_{l+N}} \right. \\ 
& \qquad + \left. \sum\limits_{j=0}^{l} \Lip(f^N) \Lip_{f^N(x)}(f^l) \sup\frac{1}{A_N} \sup(a) \sup\frac{A_N}{A_{l+N}} \right), \\
&\leqslant \varepsilon \max(\lambda^{-1}, \lambda +\beta_{\sigma} + \delta_{\sigma})^l \Lip \tfrac{1}{a} C_{\indep x,i} \left( \frac{1}{\eta} + n_V\left(\frac{1}{\lambda + \beta_{\sigma}}\right)^{n_V}  \right. \\ 
& \qquad + \left. \sup(a) \sum\limits_{j=0}^{l} \max\left( \frac{\lambda^{-1}}{\lambda + \beta_{\sigma}}, 1 + \frac{\delta_{\sigma}}{\lambda + \beta_{\sigma}} \right)^j \right).
\end{align*}
These two bounds are independent of $N$, hence of $x$.
\end{enumerate}

By setting $\beta_{\sigma} \in \, ] - \lambda + \lambda^{-2} , -\lambda -1[$, all the bounds on $\Lip_x(s_i)$ decay geometrically. Hence all partial sums of $\sum s_i$ share a common Lipschitz constant near each point of $U_\Sigma$, independent of the base-point.

We give now some estimates when $x \in K$. Therefore $f^n(x) \in A_{\eta}$ for all $n$. The following estimate holds:
\begin{align*}
\Lip_x(s_i) &\leqslant \lambda^{-i} \left( \Lip(b) \Lip(f^i) \sup\frac{1}{A_i} + \sup|b| \Lip\left(\frac{1}{a}\right) \sum\limits_{j=0}^{i} \Lip(f^j)\sup\frac{1}{A_{j-1}} \sup\frac{A_j}{A_i} \right), \\
&\leqslant \lambda^{-i} (1+\varepsilon)^2 B \left( \Lip(b) + \sup|b| \sup(a) \Lip\left(\frac{1}{a}\right) \sum\limits_{j=0}^{i} \left( \frac{1}{1+\eta} \right)^j \right), \\
&\leqslant C_{\indep x, i} \, \lambda^{-i}.
\end{align*}

Finally, every partial some of $\sum s_i$ shares a common Lipschitz constant on $S_g \smallsetminus \Sigma$. Therefore $v^s$ is Lipschitz continuous on $S_g \smallsetminus \Sigma$.
\end{proof}

\subsection{Differentiability of $v^s$}

Here we prove that when the function $k$ is $\mathcal{C}^2$, the stable vector field $v^s$ is $\mathcal{C}^1$. In order to prove this result, we use similar computations as in the proof of Theorem~\ref{thm:vs_lipschitz} and show that $v^s$ is differentiable on every compact set of $U_\Sigma$ and on $K$. We then use the relation $\mathrm{d}f \, v^s = \lambda^{-1} v^s \circ f$ (more precisely, the differential of this relation) in order to prove that there is a unique extension of $\mathrm{d}v^s$ from $U_\Sigma$ to $S_g \smallsetminus \Sigma$, and it coincides with $\mathrm{d}v^s$ on $K$.

\begin{theorem}\label{thm:vs_C1}
If the function $k : \mathbbm{R} \to \mathbbm{R} $ in the construction of $f$ is also $\mathcal{C}^2$, then the vector field $v^s$  and the flow $h_t$ are $\mathcal{C}^1$.
\end{theorem}

\begin{proof}
From the same estimates as in the proof of Theorem~\ref{thm:vs_lipschitz}, we get that the series of differentials $\sum\limits_{i \geqslant 0} \mathrm{d}s_i$ converges uniformly on $K$ and on compact subsets of $U_\Sigma \smallsetminus \Sigma$. By uniform converge, $v^s$ is therefore differentiable on $K$ and on $U_\Sigma \smallsetminus \Sigma$, but we still need to prove that $x \mapsto \D_x v^s$ is continuous on $S_g \smallsetminus \Sigma$. To this end, we use the fact that $v^s$ is uniformly contracted by $f$.

By design, $v^s$ satisfies the equality $\D_x f \, v^s(x) = \lambda^{-1} v^s(f(x))$ for all $x \notin \Sigma$. Now, by differentiation, we get for all $x$ in $U_\Sigma$
\begin{align}\label{eq:diff_condition}
\D^2_x f(v^s(x), \cdot) + \D_x f \, \D_x v^s = \lambda^{-1} \D_{f(x)} v^s \, \D_x f.
\end{align}

Let $x \in K$ and $(x_n)_n$ be a sequence in $U_\Sigma$ converging to $x$ as $n$ goes to infinity. By the Arzel\`a--Ascoli theorem, in order to prove that $(\D_{x_n} v^s)_n$ converges to $\D_x v^s$, it is sufficient to prove that $(\D_{x_n} v^s)_n$ has a unique subsequential limit.

To be exact, in order to apply the Arzel\`a--Ascoli theorem, we need the maps to have a compact domain. We address this problem by associating to any linear map $l: \mathbbm{R}^d \to \mathbbm{R}^d$ its restriction to the unit sphere $\tilde{l} : \mathbbm{S}^{d-1} \to \mathbbm{R}^d$, in addition with the closed condition 
\begin{align}\label{eq:closed_condition} ||x + \theta y|| \, \tilde{l} \left(\frac{x + \theta y}{||x + \theta y||} \right) = \tilde{l}(x) + \theta \tilde{l}(y), \qquad x,y \in \mathbbm{S}^{d-1}, \, \theta \in \mathbbm{R}.
\end{align}
Now, any linear map can be built from a map on the sphere satisfying the condition (\ref{eq:closed_condition}). This one-to-one correspondence is enough to overcome the issue of non-compactness of the domain.

Let $u_x$ be a subsequential limit of $(\D_{x_n} v^s)_n$. Using (\ref{eq:diff_condition}) and the fact that $f$ is $\mathcal{C}^2$, we also get a subsequential limit $u_{f(x)}$ of $(\D_{f(x_n)} v^s)_n$. By the same process, we get for all integer $k$ a subsequential limit $u_{f^k(x)}$ of $(\D_{f^k(x_n)} v^s)_n$ so that \[ \D^2_{f^k(x)} f(v^s(f^k(x)), \cdot) + \D_{f^k(x)} f \, u_{f^k(x)} = \lambda^{-1} u_{f^{k+1}(x)} \, \D_{f^k(x)} f. \]

Taking the difference with (\ref{eq:diff_condition}) we get, after induction, that for all integer $k$
\begin{align}\label{eq:rel_alpha_k}
\D_x f^k (\D_x v^s - u_x) = \lambda^{-k}(\D_{f^k(x)} - u_{f^k(x)}) \D_x f^k
\end{align}

We now prove that the difference $\alpha_0 \coloneqq \D_x v^s - u_x$ is the zero map. First, notice that since $v^s(x) = e_v - (\Sigma_i s_i(x)) e_h$, we must have $\mathrm{Im} (\D_x v^s) \subset \mathbbm{R}e_h = E^u(x)$ and, by taking limits, $\mathrm{Im} (u_x) \subset E^u(x)$. Therefore $\mathrm{Im}(\alpha_0) \subset E^u(x)$. Since $v^s$ is Lipschitz continuous, the operators $\D_{f^k(x)} v^s - u_{f^k(x)}$ are uniformly bounded. Therefore, from the hyperbolicity of $K$ and the relation (\ref{eq:rel_alpha_k}), we get that $\alpha_0 (\mathbbm{R}v^s(x)) \subset \mathbbm{R}v^s(x)$, and so $\alpha_0(v^s(x))=0$. Since $(v^s(x),e_h)$ is a basis of $\mathbbm{R}^2$, there exists some real number $\alpha$ such that $\alpha_0(e_h)=\alpha e_h$. Applying (\ref{eq:rel_alpha_k}) to $e_v$, we get that \[ 0 = (\alpha \mathrm{Id} - \lambda^{-k}(\D_{f^k(x)} f - u_{f^k(x)})) \D_x f \, e_h. \]
If $\alpha$ is not zero, then for large enough value of $k$ the map $(\mathrm{Id} - \frac{\lambda^{-k}}{\alpha}(\D_{f^k(x)} f - u_{f^k(x)})) \D_x f$ is invertible, hence a contradiction. Therefore $\alpha = 0$ and $u_x = \D_x v^s$. Finally, we get that $x \in U_\Sigma \smallsetminus \Sigma \mapsto \D_x v^s$ extends continuously, in a unique fashion, to $S_g \smallsetminus \Sigma$.
\end{proof}

\begin{remark}
In the case when the surface is the torus $\mathbbm{T}^2$, $v^s$ cannot be $\mathcal{C}^2$: if so the induced flow would also be $\mathcal{C}^2$, as well as its Poincar\'e map to a transverse circle. However this map is a Denjoy counterexample since it has a wandering interval, and is therefore at most $\mathcal{C}^1$ with bounded-variation derivative. It is not clear whether this bound on the regularity of $v^s$ still holds for higher genus surfaces.
\end{remark}

\subsection{Continuity of $v^s$ with respect to $\beta$}

In the next section we prove that $h_t$ is uniquely ergodic and that $f$ is mixing with respect to the invariant measure of $h_t$. To do so, we first prove that the family of vector fields $v^s$ is smooth with respect to the amplitude parameter $\beta$ in the definition of $f$. 

We will use the following notations. For all $\beta = (\beta_{\sigma})_{\sigma \in \Sigma}$, write $f_{\beta}$ the function $f$ with the amplitude parameter $\beta$, and $v^s_{\beta}$ its corresponding vector field. We also assume the parameter $(\alpha_{\sigma})_{\sigma \in \Sigma}$ to be fixed.

In this section, in order to simplify the notations, we only consider the case $\# \Sigma =1$, hence the vector $\beta$ has only one component. The general case leads to very similar computations where one must estimates the norm and Lipschitz constant on each $U_\sigma$, $\sigma \in \Sigma_{\beta} = \{ \sigma \in \Sigma \mid \beta_\sigma < 1 - \lambda$, and on $K_\beta = S_g \smallsetminus U_{\Sigma_\beta}$, where $U_{\Sigma_\beta} = \bigcup_{\sigma \in \Sigma_\beta}U_{\sigma}$. The analogous of the partition used in the proof of Theorem~\ref{thm:regularity_beta_vs} then has $2^{\# |\Sigma|}$ elements, but the argument used in order to deal with the boundaries of the elements of the partition remain the same.

More precisely, we prove the following theorem.

\begin{theorem}\label{thm:regularity_beta_vs}
The map $\beta \in \, ]-\lambda + \lambda^2,0] \mapsto v^s_{\beta}$ is continuous for the sup-norm. As a consequence, the function $(x,\beta) \mapsto v^s_{\beta}(x)$ is continuous on $\left( S_g \smallsetminus \Sigma \right) \times \, ]-\lambda + \lambda^{-2},0]$.
\end{theorem}

To show this continuity, we split the domain into three subsets. We will need the following lemma.

\needspace{5em}
\begin{lemma}\label{lemma:dim_one_eigenspace}
For all $\beta$ in $]-\lambda + \lambda^{-2},0]$, the eigenspace of $(f_{\beta })_{*} \coloneqq (\D f_{\beta})^{-1} \, U_{f_{\beta}}$ associated with the eigenvalue $\lambda$ is of dimension one when acting on the space of bounded and continuous vector fields on the tangent vector bundle of $S_g \smallsetminus \Sigma$, where $U_f$ stands for the Koopman operator of $f$.
\end{lemma}

\begin{proof}
Let $\beta \in \, ]-\lambda + \lambda^{-2},0]$. Let $w$ be a vector field in the eigenspace of $(f_{\beta })_{*}$ associated with the eigenvalue $\lambda$. In other words, $w$ is such that $\D_x f_{\beta}(w(x))=\lambda^{-1} w(f_{\beta}(x))$, for all $x$. Now, since $v^s$ is continuous, non vanishing and transverse to $e_h$, there exist two functions $w_1$ and $w_2$ uniquely determined such that $w(x) = w_1(x)v^s(x) + w_2(x) e_h$ for all $x$. These two functions are bounded and continuous. Hence, we have,
\begin{align*}
\D_x f_{\beta} (w_2(x) e_h) &= a(x)w_2(x)e_h = \D_x f_{\beta}(w(x) - w_1(x)v^s(x)), \\
&= \lambda^{-1}(w(f_{\beta}(x)) - w_1(x)v^s(f_{\beta}(x))),\\
w(f_{\beta}(x)) &= w_1(x) v^s(f_{\beta}(x)) + \lambda a(x) w_2(x) e_h.
\end{align*}
Therefore, $w_1$ is invariant by $f_{\beta}$ and for all $i>0$,
\begin{align*}
w_2(x) = \prod\limits_{j=0}^{i-1}\frac{1}{\lambda a(f_{\beta}^j(x))}w_2(f_{\beta}^i(x)).
\end{align*}

By continuity of $w_2$ and compactness of $S_g$, $w_2$ is bounded. Now, we distinguish two cases in order to prove that $w_2=0$.

For $\beta_{\sigma} < 1-\lambda$, there exists a fixed point $p^{\sigma}_{i}$, in $K$, whose unstable leaf is dense. Since at this point $a(p^{\sigma}_{i})>1$, by continuity of $a$, we get that $a>1$ in a neighbourhood of $p^{\sigma}_{i}$, hence $1/(\lambda a) < \lambda^{-1} < 1$ and $w_2=0$.

For $1-\lambda \leqslant \beta_{\sigma} \leqslant 0$, we know that the unstable leaf of $\sigma$ is dense in $S_g$. By continuity on every leaf of the branched cover at $\sigma$, we can set $a(\sigma) = \lambda + \beta_{\sigma} \geqslant 1$. Hence, in a neighbourhood of $\sigma$, we get $1/(\lambda a) \leqslant \lambda^{-1} <1$, hence $w_2=0$.
\vspace{1em}

In order to prove that $w_1$ is constant, we also distinguish two cases. 

For $\beta_{\sigma} < 1-\lambda$, the unstable leaf of each $p^{\sigma}_{i}$ is dense. Hence $w_1(x) = w_1(p^{\sigma}_{i})$ for all $x$. Hence the claim in this case.

For  $1-\lambda \leqslant \beta_{\sigma} \leqslant 0$, the unstable leaf of $\sigma$ is dense. Therefore, $w_1(x) = w_1(\sigma)$ for all $x$. Hence the claim.
\end{proof}

\begin{proof}[Proof of Theorem~\ref{thm:regularity_beta_vs}.]
We first prove that $|| v^s_{\beta} - v^s_{\beta_0} ||_{\infty} \xrightarrow[\beta \to \beta_0]{} 0$ for all $\beta_0 $ in $ ]-\lambda + \lambda^{-2}, 1- \lambda[$. From proofs of Theorems~\ref{thm:vs_bounded_continuous} and \ref{thm:vs_lipschitz}, we can see that on a small enough neighbouhood $B_0$ of $\beta_0$, the vector fields $v^s_{\beta}$ are uniformly bounded, as well as their Lipschitz constants. By the Arzel\`a-Ascoli theorem, the set $\{v^s_{\beta} \mid \beta \in B_0 \}$ is relatively compact. Take a sequence of $(\beta_n)_n$ converging to $\beta_0$, then every sub-sequential limit $w$ of $(v^s_{\beta_n})_n$ must satisfies $(f_{\beta_0})_* w = \lambda w$. By Lemma~\ref{lemma:dim_one_eigenspace}, the space of such vector fields is one dimensional, hence there exists a constant $c$ such that $w = c v^s_{\beta_0}$. Since in the basis $(e_h,e_v)$ all the component of $v^s_{\beta}$ along $e_v$ is $1$, we get that $c=1$. Hence $v^s_{\beta_n}$ converges uniformly to $v^s_{\beta_0}$, and so for all sequences $(\beta_n)_n$. The rest of the claim follows directly by the triangle inequality and Lipschitz continuity.

\vspace{1em}
We now prove that $|| v^s_{\beta} - v^s_{\beta_0} ||_{\infty} \xrightarrow[\beta \to \beta_0]{} 0$ for all $\beta_0 \in [ 1- \lambda, 0]$. 
The same argument as in the case above holds. Indeed, for all $\beta \in [ 1- \lambda, 0]$ we get
\begin{align*}
\sum\limits_{i \geqslant 0} \left| \lambda^{-i} b_{\beta} \circ f_{\beta}^i \prod\limits_{j=0}^{i}\frac{1}{a_{\beta} \circ f_{\beta}^j} \right| &\leqslant \frac{1}{1-\lambda^{-1}} ||b_{\beta}||_{\infty}.
\end{align*}
Hence $v^s_{\beta}$ is uniformly bounded for $\beta$ in a neighbourhood of $\beta_0$.
Similarly, the following estimate on the Lipschitz constant holds for all $\varepsilon >0$
\begin{align*}
\sum\limits_{i \geqslant 0} \Lip_x \left( \lambda^{-i} b_{\beta} \circ f_{\beta}^i \prod\limits_{j=0}^{i}\frac{1}{a_{\beta} \circ f_{\beta}^j} \right) &\leqslant (1+\varepsilon)^2 ||v_{\beta}^s||_{\infty} \sum\limits_{i \geqslant 0} \lambda^{-i} \biggl( \Lip(b_{\beta})  \\
& \quad  + i \Lip \left(\frac{1}{a_{\beta}} \right) ||a_{\beta}||_{\infty} ||b_{\beta}||_{\infty} \biggr)
\end{align*}

\vspace{1em}
Finally, we prove that $|| v^s_{\beta} - v^s_{1 - \lambda} ||_{\infty} \xrightarrow[\beta \to (1 - \lambda)^-]{} 0$. 
Recall notations from Proposition~\ref{prop:open_cover} and let $V$ be a neighbourhood of some $\sigma \in \Sigma$ as in the proof of Theorem~\ref{thm:vs_bounded_continuous}. Let $x \in f^{-N}(V) \cap U_{\sigma}$ and let $n(x)$ be the number of points in the orbit of $x$ that belong to $B_{\sigma,\delta} \smallsetminus V$. Then $N-n(x) \geqslant 0$ and we have the following estimates depending on $i$:
\begin{enumerate}[label=$\bullet$, wide]
\item if $i \leqslant N - n(x)$, then $|s_i(x)| \leqslant \lambda^{-i} \left( \frac{1}{1+\eta} \right)^{i+1} \sup|b|$.
\item if $N-n(x) < i \leqslant N$, then $|s_i(x)| \leqslant \lambda^{-i} \left( \frac{1}{1+\eta} \right)^{N-n(x)} \left( \frac{1}{\lambda + \beta} \right)^{i-(N-n(x))} \sup|b| $ so that $|s_i(x)| \leqslant \sup|b| \lambda^{-i} \left( \frac{1}{\lambda + \beta} \right)^{i}$.
\item if $i = j + N > N$, then $|s_i(x)| \leqslant \lambda^{-(j+N)} \left( \frac{1}{\lambda + \beta } \right)^{j+N} \Lip(b) \varepsilon \max(\lambda^{-1}, \lambda + \beta_{\sigma} + \delta_{\sigma}) ^j$.
\end{enumerate}
Therefore, $\sum\limits_{i \geqslant 0} |s_i(x)| \leqslant ||b||_{\infty} \left( \frac{\lambda}{\lambda - 1} + \frac{\lambda(\lambda + \beta)}{\lambda(\lambda + \beta) - 1} + \varepsilon \frac{\Lip(b)}{1- \frac{\max(\lambda^{-1},\lambda + \beta + \delta)}{\lambda(\lambda + \beta)}} \right)$, and so for all $\varepsilon >0$.

Hence, the family of vector fields $(v^s_{\beta})_{\beta}$ is uniformly bounded on $S_g$ and the bound can be chosen uniformly in $\beta$ for $\beta \in [1- \lambda - \varepsilon, 1 - \lambda]$. However, the estimates we had on the Lipschitz constants are no longer good enough to apply the same argument as in above cases.

Let $x \in S_g$ and $(x_n,\beta_n)_n$ be a sequence converging to $(x,1-\lambda)$ and such that $\beta_n < 1-\lambda$ for all $n$. For $n$ large enough, the sequence $(v_{\beta_n}^s(x_n))_n$ is bounded and let $w(x)$ be a sub-sequential limit. Since for all $k \geqslant 0$, the sequence $(v_{\beta_n}^s(f^k_{\beta_n}(x_n)))_n$ is bounded, by a diagonal argument we can assume up to extracting that the sequences converge to some vectors $w(f^{k}_{1-\lambda}(x))$. By continuity of $\D f_{\beta}$ in $\beta$, we get that $\D_x f^{k}_{1-\lambda} w(x) = \lambda^{-k} w(f^k_{1-\lambda}(x))$ for all $k$. By expressing vectors $w(f^{k}_{1-\lambda}(x))$ in the basis $(v^s_{1-\lambda}(x), e_h)$, we see that $w(x) \in \mathbbm{R}v^s_{1-\lambda}(x)$. Since each vector of the form $v_{\beta_n}^s(f^k_{\beta_n}(x_n))$ has a component equal to $1$ along $e_v$ in the basis $(e_h,e_v)$, we get $w(x) = v^s_{1-\lambda}(x)$. Hence $(x,\beta)\mapsto v^s_{\beta}(x)$ is continuous at $(x,(1-\lambda)^-)$. 

Now, suppose that $|| v^s_{\beta} - v^s_{1 - \lambda} ||_{\infty}$ does not converge to zero as $\beta$ converges to $1-\lambda$ from below. Then, there exists some positive $\varepsilon$ and sequences $(\beta_n)_n$ and $(x_n)_n$ such that $\lim\limits_{n \to \infty} \beta_n = (1-\lambda)^-$ and $||v^s_{\beta_n}(x_n) - v^s_{1-\lambda}(x_n)|| \geqslant \varepsilon$. Up to extracting, we can assume that $(x_n)_n$ converges to some $x$. Therefore $||v^s_{\beta_n}(x_n) - v^s_{1-\lambda}(x)|| \geqslant \varepsilon/2$ for large enough $n$. This contradicts the continuity of $(x,\beta)\mapsto v^s_{\beta}(x)$ at $(x,(1-\lambda)^-)$. 

\vspace{1em}
The continuity of $(x,\beta) \mapsto v^s_{\beta}(x)$ on $(S_g \smallsetminus \Sigma) \times ] -\lambda + \lambda^2 , 0]$ follows from \[ ||v^s_{\beta}(x) - v^s_{\beta_0}(x_0)|| \leqslant ||v^s_{\beta} - v^s_{\beta_0}||_{\infty} + ||v^s_{\beta_0}(x) - v^s_{\beta_0}(x_0)|| \xrightarrow[(x,\beta) \to (x_0,\beta_0)]{} 0. \]
\end{proof}

\subsection{Renormalized flow and topological properties of $K$}\label{subsect:h_t}

Since $v^s$ is Lipschitz continuous, we can integrate it into a flow $h_t$. Since some trajectories reaches in finite time conical points, for which $v^s$ is not defined, this flow must be treated carefully. On the other hand, since $v^s$ is uniformly contracted by the action of $f$, $h_t$ is renormalized by $f$. From this relationship between $f$ and $h_t$, we can deduce further topological properties about stable leaves and the set $K$. We first prove that for each fixed hyperbolic point $p^{\sigma}_i$, its stable leaf coincides with the orbit by $h_t$ of this point. From this fact and Proposition~\ref{prop:K_as_Finite_Union}, we deduce that $K$ is transverse to any horizontal leaf. We then show that $K$ is in fact equal to the closure of the stable leaf of any hyperbolic fixed point $p^{\sigma}_i$, hence $K$ is connected. Finally, we prove that $f$ is topologically transitive with respect to the trace topology of $S_g$ on $K$.

\begin{proposition}\label{prop:commutation_relation_and_stability_of_K}
For all $x \in S_g \smallsetminus \Sigma$ and $t$ for which $h_t(f(x))$ is well defined, $f$ and $h_t$ satisfy the relation, \[ f \circ  h_{\lambda t}(x) = h_t \circ f(x). \]
The orbit $\{h_t(x)\}$ of any point $x$ in $K$ is well defined for all time $t$. Furthermore, for all $t \in \mathbbm{R}$, $h_t(K) = K$.
\end{proposition}

\begin{proof}
Since $\diff_x f(v^s(x)) = \lambda^{-1} v^s(f(x))$, for $x \in S_g \smallsetminus \Sigma$, notice that, 
\begin{align*}
\frac{\D}{\D t} (f \circ h_{\lambda t}(x)) = \diff_{h_{\lambda t}(x)} f \left( \frac{\D}{\D t} h_{\lambda t}(x) \right) = \diff_{h_{\lambda t}(x)} f ( \lambda v^s(h_{\lambda t}(x))) = v^s(f \circ h_{\lambda t}(x)).
\end{align*}
In particular, $t \mapsto f(h_{\lambda t}(x))$ solves $\frac{\D}{\D t} g = v^s(g)$ with initial condition $g(0) = f(x)$ at $t=0$. By construction of the flow $h_t$, $t \mapsto h_t(f(x))$ solves the same differential equation, with the same initial condition.
Therefore, by uniqueness of the solution (since $v^s$ is Lipschitz) the two functions $t \mapsto f(h_{\lambda t}(x))$ and $t \mapsto h_t(f(x))$ must coincide where they are well defined. In other words, $f \circ  h_{\lambda t} = h_t \circ f$ for all $t$ where the solution is defined.

Let $\mathcal{F} \coloneqq S_g \smallsetminus( \Sigma \cup \{ x \in S_g \smallsetminus \Sigma \mid  \forall t \in \mathbbm{R}, \, h_t(x) \text{ exists} \})$ be the set of points whose trajectory are not well defined for all time. We now prove that if $x \in \mathcal{F}$, then there exist $\sigma \in \Sigma$ and $t_0 \in \mathbbm{R}$ such that $h_t(x) \to \sigma$ as $t$ tends to $t_0$. Indeed, by compactness of $S_g$, up to taking a sub-sequence $(t_n)_n$ that converges to $t_0$, the limit of $(h_{t_n}(x))_n$ exists. If this limit doesn't belong to $\Sigma$, we can extend the solution past $t_0$.

To prove that $h_t$ is complete when restricted to $K$, it suffices to prove that $K \subset \mathcal{F}^c$, or equivalently, that $\mathcal{F} \subset U_{\Sigma}$. By contradiction, let $x \in \mathcal{F} \cap K$. Let $t_0$ and $\sigma$ be as above. Hence, the smooth curves $f^n \circ h_{t}(x) : t \in [0,t_0] \to S_g$ join $K$ to $\Sigma$ and their lengths are less than $\lambda^{-n}t_0 ||v^s||_{\infty}$. This contradicts the fact that $\D(K, \Sigma) > \min \{ |p^{\sigma}| \mid \sigma \in \Sigma \} > 0$ by Proposition~\ref{prop:existence_ball_in_basin}.

Since $\mathcal{F} \cap K = \emptyset$, $h_t(x)$ is well defined for all $x \in K$ and all time $t$. Let $x\in K$. By contradiction, assume there exists $t_1$ such that $h_{t_1}(x) \in U$. Therefore $f^n(h_{t_1}(x))$ converges to some $\sigma$ as $n$ goes to infinity and the curves $f^n \circ h_{t}(x) : t \in [0,t_1] \to S_g$ joins $K$ to some arbitrarily close point to $\sigma$ for $n$ large enough. Since such a curve is of length at most $\lambda^{-n}t ||v^s||_{\infty}$, it contradicts $\D(K, \Sigma) >0$.
\end{proof}

This commutation relation between $f$ and $h_t$ is a central argument throughout the rest of this article.

We can now deduce the announced topological properties of the invariant leaves and of $K$.

\begin{proposition}\label{prop:stable_leaf_is_flow_trajectory}
For all $p_{i}^{\sigma}$, we have the equality of sets $W^{ss}(p_{i}^{\sigma}) = h_{\mathbbm{R}}(p_{i}^{\sigma})$. Also, the set $K$ is transverse to any horizontal leaf.
\end{proposition}

\begin{proof}
Let $t \in \mathbbm{R}$. Hence $f^n(h_t(p_{i}^{\sigma}))=h_{\lambda^{-n}t}(p_{i}^{\sigma})$ converges to $p_{i}^{\sigma}$ as $n$ goes to infinity. Hence $h_{\mathbbm{R}}(p_{i}^{\sigma}) \subset W^{ss}(p_{i}^{\sigma})$. By the commutation relation between $f$ and $h_t$, we get that $h_{\mathbbm{R}}(p_{i}^{\sigma})$ is invariant by $f$. In the linearisation near $p_{i}^{\sigma}$ given by the Grobman--Hartman theorem, the only invariant part by $f$ corresponds to a small piece $\gamma$ of the stable leaf of $p_{i}^{\sigma}$. By invariance of $h_{\mathbbm{R}}(p^{\sigma}_i)$ by $f$, we get $\gamma \subset h_{\mathbbm{R}}(p_{i}^{\sigma})$. Finally, since $W^{ss}(p_{i}^{\sigma}) = \bigcup_{n\geqslant 0} f^{-n}(\gamma)$, we get $h_{\mathbbm{R}}(p_{i}^{\sigma}) = W^{ss}(p_{i}^{\sigma})$.

Since the convergence of the infinite sum defining $v^s$ is uniform on $K$, the horizontal component of the vector field $v^s$ is continuous, hence bounded. Therefore, all the stable leaves $W^{ss}(p_{i}^{\sigma})$ are transverse to any horizontal leaf. The result holds by taking the closure since slopes are bounded and by Proposition~\ref{prop:K_as_Finite_Union}.
\end{proof}

\begin{theorem}\label{thm:K_connected}
The set $K$ is connected and it can be written as $K= \overline{W^{ss}(p^{\sigma}_i)}$, for any $\sigma \in \Sigma$ and any $1 \leqslant i \leqslant 2 n_{\sigma}$.
\end{theorem}

\begin{proof}
Let $\sigma_1, \, \sigma_2 \in \Sigma$ and $i_1,i_2$ be two integers. For simplicity, call $p_1=p^{\sigma_1}_{i_1}$ and $p_2=p^{\sigma_2}_{i_2}$. Let $W_2$ be the open set containing $p_2$ given by the Grobman--Hartman theorem -- without loss of generality we assume it is a rectangle with horizontal and vertical sides. Since $W^{su}(p_2)$ contains a dense horizontal leaf, and $W^{ss}(p_1)$ is transverse with all vertical leaves, the intersection $W^{su}(p_2) \cap  W^{ss}(p_1)$ is non-empty. Let $x \in W^{su}(p_2) \cap  W^{ss}(p_1)$ and let $\gamma$ be a small connected piece of $W^{ss}(p_1)$ containing $x$ in its interior. Then, for large enough $n \geqslant 0$, we see that $f^{-n}(\gamma) \cap W_2$ accumulates on $W^{ss}(p_2) \cap W_2$. Therefore, $W^{ss}(p_2) \cap W_2 \subset \overline{W^{su}(p_2) \cap  W^{ss}(p_1)} \subset \overline{W^{ss}(p_1)}$. Since $\overline{W^{ss}(p_1)}$ is invariant by the action of $f$ and $W^{ss}(p_2) = \bigcup\limits_{n \geqslant 0} f^{-n}(W^{ss}(p_2) \cap W_2)$, we get the inclusion $\overline{W^{ss}(p_2)} \subset \overline{W^{ss}(p_1)}$. Since the choice of $p_1$ and $p_2$ is arbitrary, the result follows from Proposition~\ref{prop:K_as_Finite_Union}.
\end{proof}

\begin{theorem}\label{thm:f_transitive_on_K}
The function $f:K \to K$ is transitive with respect to the trace topology of $S_g$ on $K$.
\end{theorem}

\begin{proof}
Let $U_1$ and $U_2$ be open sets in $S_g$ that have non-empty intersection with $K$. Let $p_1 = p^{\sigma_1}_{i_1}$ and $p_2 = p^{\sigma_2}_{i_2}$ for some $\sigma_1, \, \sigma_2 \in \Sigma$ such that $U_i \cap ( W^{ss}(p_i) \cap W^{su}(p_i)) \neq \emptyset$ for $i=1,2$. Since $W^{ss}(p_2)$ is transverse with all the horizontal leaves, we can find a rectangle $V_2$ contained in $U_2$ whose sides are vertical and horizontal, such that $W^{ss}(p_2)$ crosses $V_2$ \emph{from side to side}.

By density of $W^{su}(p_1)$, there exists $x_2 \in V_2 \cap W^{su}(p_1)$. Let $W_1$ be the open set of linearisation near $p_1$ -- without loss of generality, we can assume $W_1$ to be a rectangle with horizontal and vertical sides. For large enough $n \geqslant 0$, the set $f^{-n}(V_2)$ crosses \emph{vertically} $W_1$.

Let $x_1 \in U_1 \cap W^{ss}(p_1)$ and $\varepsilon >0$ be such that the horizontal segment $\gamma$ of length $\varepsilon$, containing $x_1$ in its interior, is contained in $U_1$. For all large enough $m \geqslant 0$, the line $f^m(\gamma)$ crosses \emph{horizontally} $W_1$. Hence $f^m(U_1) \cap f^{-n}(U_2) \neq \emptyset$.
\end{proof}

It easily follows from the transitivity of $f$ and the closing lemma that periodic points of $f$ are dense in $K$. Therefore $K$ is an Axiom A attractor in the sense of \cite{Ruelle1976measure}.

\begin{theorem}\label{thm:axiom_A}
If $\Sigma^{\varepsilon}$ is an open $\varepsilon$-neighbourhood of $\Sigma$ for some small enough $\varepsilon >0$, $U \coloneqq S_g \smallsetminus \overline{\Sigma^{\varepsilon}}$ and $f^{-1}$ is $\mathcal{C}^2$ away from $\Sigma$, then $K$ is an Axiom A attractor for $f^{-1} : U \to U$.
\end{theorem}

\section{The induced GIET}\label{sect:f_is_mixing}

In this part we construct a GIET $T$ as the Poincar\'e map of $h_t$ to some transversal segment, and we prove that it satisfies the conclusion of Theorem~\ref{thm:main_1}. For the semi-conjugacy, it is sufficient to prove --~thanks to a result by Yoccoz \cite{yoccoz2005echanges}~-- that $T$ follows the same orbit as a self-similar IET when renormalized by the Rauzy--Veech algorithm. To do so, we construct multiple partitions into rectangles of $S_g$. Minimality and unique ergodicity of $T$ then follow from the one of the semi-conjugated self-similar IET. Since $h_t$ is the suspension flow over $T$, $h_t$ is also uniquely ergodic, of unique invariant measure $\mu$, supported by $K$. Because of the commutation relation (\ref{eq:commut_flow}) between $f$ and $h_t$, the measure $\mu$ is also invariant by $f$. We prove that $f$ is mixing with respect to $\mu$.

\begin{theorem}\label{thm:unique_ergodicité_(h_t)}
The flow $h_t$ is uniquely ergodic. Furthermore the support of the invariant measure is $K$.
\end{theorem}

\begin{corollary}\label{corol:f_is_mixing}
The unique invariant measure $\mu$ of $h_t$ is also invariant by $f$, and $f$ is mixing with respect to $\mu$.
\end{corollary}

This theorem and its corollary are a restatement of Theorem~\ref{thm:main_2}(iii).

\begin{proof}[Proof of Corollary~\ref{corol:f_is_mixing}.]
Since $K$ is invariant by $f$ and by the flow $h_t$ and since $h_t$ is well defined for all $t$ on $K$, we have
\begin{align*}
f_* \mu &= f_* ((h_t)_* \mu) = (f \circ h_t)_* \mu = (h_{\lambda^{-1}t})_* (f_* \mu).
\end{align*}
Therefore the measure $f_* \mu$ is invariant by the flow $h_t$. By unique ergodicity of the flow, we must have $f_* \mu = \mu$.

Let $F \in \L2$  be such that $\int F \, \D\mu = 0$. We now prove that the sequence $(F \circ f^n)_n$ weakly converges to zero. By invariance of the measure, the sequence is bounded in the $\L2$ norm. By the Banach-Alaoglu-Bourbaki theorem, this sequence lives in a weakly compact set. Let $\bar{F}$ be a sub-sequential weak limit of $(F \circ f^n)_n$ and let $(n_k)_k$ be a strictly increasing sequence of integers such that $F \circ f^{n_k} \xrightharpoonup[k \to \infty]{} \bar{F}$. On the other hand,
\begin{align*}
||F \circ f^{n_k} \circ h_t - F \circ f^{n_k}||_{L^2} &= ||F \circ h_{\lambda^{-n_k} t} \circ f^{n_k} - F \circ f^{n_k}||_{L^2},\\
&= ||F \circ h_{\lambda^{-n_k} t} - F ||_{L^2} \xrightarrow[k \to \infty]{} 0,
\end{align*}
where the final limit follows from the density of continuous functions in $\L2$. Now, $F \circ f^{n_k} \circ h_t - F \circ f^{n_k}$ converges weakly to $\bar{F} \circ h_t - \bar{F}$. The identification of the strong limit with the weak limit gives $\bar{F} \circ h_t - \bar{F}=0$. By unique ergodicity of $(h_t)_t$, $\bar{F}$ is constant. By integration, this constant is zero. Hence all the sub-sequential weak limit of $(F \circ f^n)_n$ are $0$, which proves the mixing.
\end{proof}

In order to prove Theorem~\ref{thm:unique_ergodicité_(h_t)}, we heavily rely on the semi-conjugacy result from \cite[Proposition 7]{yoccoz2005echanges}, more precisely if an IET and a GIET have the same combinatorial datum and follow a same full path in the Rauzy diagram -- when renormalized by the Rauzy--Veech algorithm -- then there exists a continuous, increasing and surjective function that semi-conjugates the two transformations.

\subsection{Construction of a GIET and $h_t$ as its suspension flow}\label{sect:construct_rect_decomp}

Recall some notation from Section~\ref{sect:first_def_perturb}. Let $\varphi$ be the pseudo-Anosov map that we perturbed in order to get $f$. By construction, $\varphi$ fixes each conical point and each separatrix. Let $\sigma \in \Sigma$ be a conical point and $\gamma_0$ be a segment of a horizontal separatrix starting at $\sigma$. From a general property of the pseudo-Anosov maps, there exists a decomposition in rectangles $\mathcal{R}_0=(R^0_1, \ldots, R^0_{\ell})$ of $S_g$ such that (up to shortening $\gamma_0$) the \emph{bases} of these rectangles form a partition of $\gamma_0$ (this is a particular case of the decomposition into rectangles and cylinders from \cite[Proposition~5.3.4]{hubbard2016teichmuller}).

Denote by $\partial_v \mathcal{R}_0$ (resp. $\partial_h \mathcal{R}_0$) the vertical (resp. horizontal) components of $\bigcup_i \partial R_i$. By construction, $\partial_h \mathcal{R}_0 = \gamma_0$. Now, $\partial_v \mathcal{R}_0$ is made of portions of trajectories for the vertical flow associate to $\varphi$ that connect a conical point to $\gamma_0$, but don't intersect $\gamma_0$ at some other previous time. 

Since the family of vector fields $(x,\beta) \mapsto v^s_{\beta}(x)$ is continuous, we can deform by some homotopy $\mathcal{R}_0$ into $\mathcal{R}_{\beta}=(R^{\beta}_1, \ldots, R^{\beta}_{\ell})$ while preserving the horizontal direction, where $\beta$ is the amplitude of the perturbations in the construction of $f$. In more details, the homotopy sends the portions of trajectories of the vertical flow that connect conical points to $\gamma_0$, to the portions of trajectories of $h_t$ which contain a conical point. Since the vector field $v^s_{\beta}$ has its vertical component constant equal to 1, these latter trajectories are the ones connecting conical points to $\gamma$, where $\gamma$ is a slightly longer or shorter copy of $\gamma_0$. Since any two trajectories do not intersect, these portions of trajectories of $h_t$ are still the shortest ones that connect conical points to $\gamma$

\begin{figure}
\begin{center}
\begin{tabular}{ccc}
\includegraphics[width=0.40\linewidth]{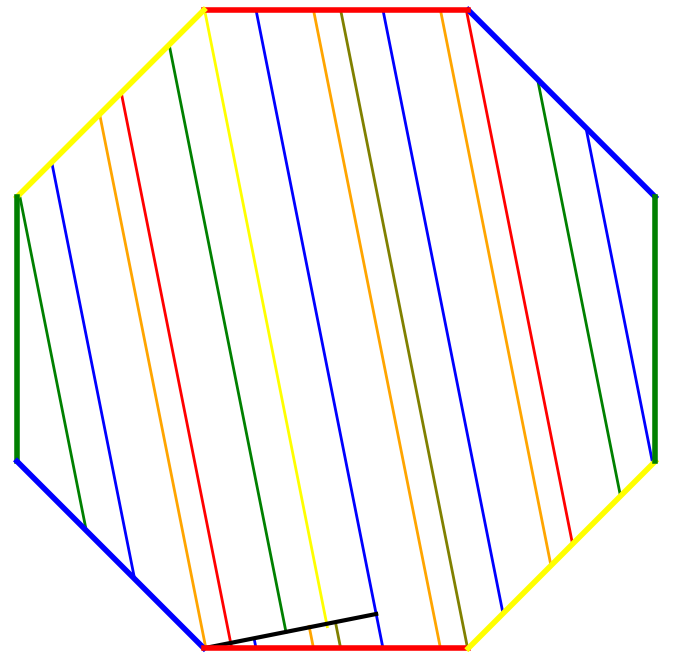} & &
\includegraphics[width=0.40\linewidth]{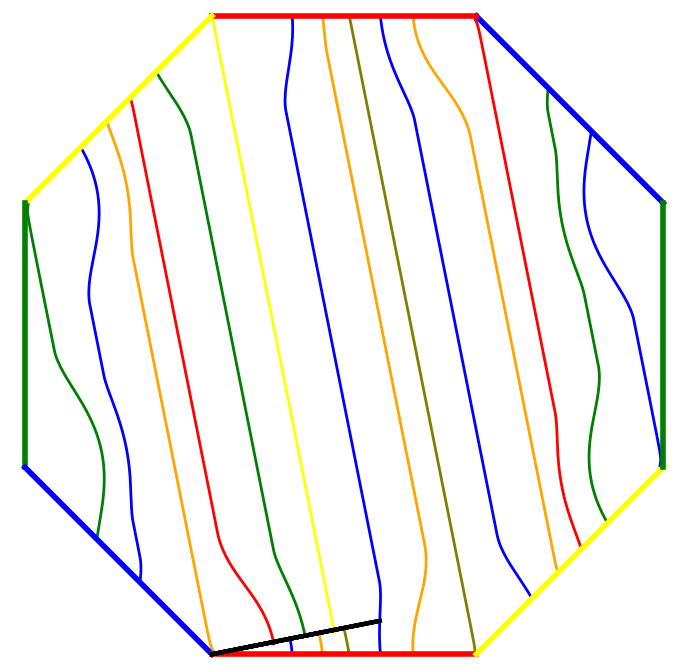} \\
Rectangle decomposition $\mathcal{R}_0$ & & Perturbed rectangle decomposition $\mathcal{R}_{\beta}$
\end{tabular}
\caption{\label{fig:zip_rect_octagon} Rectangle decompositions in the case of a flat genus two surface. The pseudo-Anosov transformation on this surface is explicited in the appendix of \cite{sinai2005weak} as the composition of an upper triangular matrix with its transpose matrix. The contracting and expanding directions of $\varphi$ match (up to a small rotation) with respectively the vertical and horizontal directions of the figures.}
\end{center}
\end{figure}

\begin{figure}
\begin{center}
\begin{tabular}{ccc}
\includegraphics[width=0.40\linewidth]{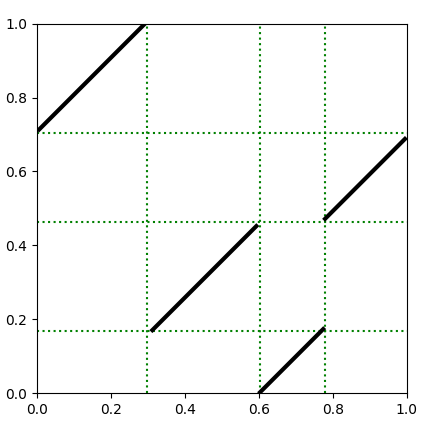} & &
\includegraphics[width=0.40\linewidth]{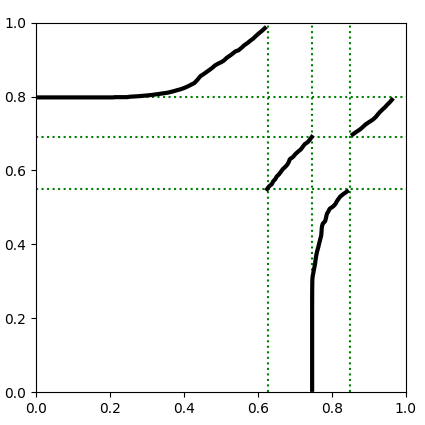} \\
Graph of the IET associated to $\mathcal{R}_0$ & & Graph of the GIET associated to $\mathcal{R}_{\beta}$
\end{tabular}
\caption{\label{fig:IET_and_GIET_associated} Graphs of the induced IET and GIET induced respectively by the rectangle decompositions in Figure~\ref{fig:zip_rect_octagon} -- both flows are going \emph{``downward"}.}
\end{center}
\end{figure}

Call $T$ (resp. $T_0$) the Poincar\'e first return map to $\gamma$ (resp $\gamma_0$) of $h_t$ (resp. of the unit speed vertical flow associate to $\varphi$). It is clear from the construction that $T_0$ is an IET and that $T$ is a GIET. Since for all $\beta$, the vertical component of $v^s_{\beta}$ is equal to one, both $T$ and $T_0$ have the same combinatorial data. Furthermore, by construction, $T$ and $T_0$ have the same path in the Rauzy-graph: fix an integer $n \geqslant 1$, then because of the continuity with respect to $\beta$, there exists a neighbourhood of $\beta_0$ such that all of the induced GIET have the same first $n$ steps in the Rauzy-Veech algorithm. Consider the maximal open set with this property, and let $\beta^*$ be a point in its boundary. We distinguish two cases: assume first that for $\beta^*$ the algorithm doesn't stop at the $n$-th step. Then there is a neighbourhood of $\beta^*$ contradicting the maximality. In the remaining case, the algorithm stops at the $n$-th step: the induced GIET must have a connection. This corresponds geometrically to a side of a rectangle of $\mathcal{R}_{\beta^*}$ connecting a conical point to another one: this is impossible since $f_{\beta^*}$ would contract this curve. Thus $T$ and $T_0$ have the same first $n$ steps in the Rauzy-Veech algorithm, for any $n$.

Since foliations associated to a pseudo-Anosov have no closed leaf (see \cite{farb2011primer}), it follows that $T_0$ has no connection, hence, by \cite{yoccoz2005echanges}, the path of $T_0$ in the Rauzy graph is full and so $T_0$ and $T$ are semi-conjugated by a continuous, increasing and surjective function. Also, since $T_0$ has no connection, it is minimal.

We summarize all this in the following proposition.

\needspace{10em}
\begin{proposition}\label{prop:semiconjugacy}
If $\sigma$ is a conical point, there exist two portions of a same  horizontal separatrix (both containing $\sigma$) $\gamma_0$ and $\gamma$, and maps $T:\gamma \to \gamma$, $T_0:\gamma_0 \to \gamma_0$ such that:
\begin{enumerate}[label=(\roman*)]
\item $T_0$ is an IET and $T$ is a GIET,
\item $T_0$ is the Poincar\'e first return map of the vertical flow associated to $\varphi$,
\item $T$ is the Poincar\'e first return map of the flow $h_t$ associated to $f$,
\item $T_0$ and $T$ have the same combinatorial data, and the same path in the Rauzy-graph,
\item there exists a continuous, increasing and surjective function $h$ such that $h \circ T = T_0 \circ h$.
\end{enumerate}
\end{proposition}

\subsection{Minimality of the flow on $K$}\label{sect:minimality_flow}

In this part we prove that the map $T$ -- from which $h_t$ is the suspension flow -- is minimal on its nonwandering set. To do so, we rely on the analysis carried out in \cite{yoccoz2005echanges}. From this, we deduce that the flow $h_t$ acts minimally on $K$ -- actually, we also prove that $K$ is an attractor for positive and negative times. This property will be useful to prove that the support of the unique invariant measure of $h_t$ is $K$.

As in \cite{yoccoz2005echanges}, define $S(\infty)$ as the union of the forward orbit of the discontinuity points of $T^{-1}$ and the backward orbits of the discontinuity points of $T$. Similarly, define $S_0(\infty)$ from the discontinuities of $T_0$ and $T_0^{-1}$. By construction, $h$ is an increasing bijection from $S(\infty)$ to $S_0(\infty)$.

Define $\Omega$ as the set of non-isolated points of $\overline{S(\infty)}$. Clearly, $\Omega$ is a closed set. We now prove that $T$ is minimal on $\Omega$.

\begin{theorem}\label{thm:T_is_minimal}
When restricted to the set $\Omega$, $T$ is minimal.
\end{theorem}

\begin{proof}
We first prove that there exists a decomposition of $\Omega$ in closed sets $\Omega = \Omega_+ \cup \Omega_-$ such that $T( \Omega_+) \subset \Omega_+$ and $T^{-1}(\Omega_-) \subset \Omega_-$. 

Let $S(\infty)_+$ be the forward orbits by $T$ of the discontinuity points of $T^{-1}$ and similarly $S(\infty)_-$ be the set of the backward orbits by $T$ of the discontinuity points of $T$. By definition of $S(\infty)$, $S(\infty) = S(\infty)_+ \cup S(\infty)_+$. Define $\Omega_{\pm}$ as the set of non-isolated points of $\overline{S(\infty)}_{\pm}$. These sets satisfy the claim.

Let $x$ be a point of $\Omega$. Up to considering its backward orbit, we assume that $x \in \Omega_+$. We want to prove that $(T^n(x))_{n\geqslant 0} $ is dense in $\Omega$. By contradiction, let $U$ be an open set such that $U \cap \Omega \neq \emptyset$ and $T^n(x) \notin U\cap \Omega$ for all $n$. Since $\Omega_+$ is stable  by the action of $T$, we can relax the last condition by $T^n(x) \notin U$ for all $n \geqslant 0$.

Since $U \cap \Omega \neq \emptyset$, $U\cap \Omega$ contains at least two different points of $S(\infty)$, therefore $h$ is not constant on $U$. Hence $h(U)$ has a non-empty interior. Finally, since the sequence $h\circ T^n(x) = T_0^n(h(x))$ avoids an open set and $T_0$ is minimal, we get a contradiction.
\end{proof}

In order to prove that $\Omega$ is an attractor for both $T$ and $T^{-1}$, we need the following three technical lemmas.

\begin{lemma}\label{lemma:h_constant_cc}
The function $h$ such that $h \circ T = T_0 \circ h$ is constant on the connected components of $\gamma \smallsetminus \overline{S(\infty)}$.
\end{lemma}

\begin{proof}
By contradiction, let $]j_-,j_+[$ be a connected component of $\gamma \smallsetminus \overline{S(\infty)}$ on which $h$ is not constant. Therefore $h(j_-) < h(j_+)$. By density of $S_0(\infty)$ in $\gamma_0$, there exist infinitely many points of $S_0(\infty)$ in the middle third segment of $[h(j_-),h(j_+)]$. Since $h : S(\infty) \to S_0(\infty)$ is a bijection, the image by $h^{-1}$ of all these points of $S_0(\infty)$ is relatively compact in $]j_-,j_+[$. Hence, there exist accumulation points of $S(\infty)$ in $]j_-,j_+[$, which is a contradiction.
\end{proof}

\begin{lemma}\label{lemma:permuted_without_cycle}
The connected components of $\gamma \smallsetminus \overline{S(\infty)}$ are permuted without cycle by $T$.
\end{lemma}

\begin{proof} 
By construction of $S(\infty)$, $T$ and $T^{-1}$ are continuous on each connected componant of $\gamma \smallsetminus \overline{S(\infty)}$. If $J$ is a connected componant of $\gamma \smallsetminus \overline{S(\infty)}$, then it is easy to see that $T(J)$ is a subset of a connected componant of $\gamma \smallsetminus \overline{S(\infty)}$. The same argument applied with $T^{-1}$ proves that the connected componants are permuted by the action of $T$.

By contradiction, let $J$ be a connected component of $\gamma \smallsetminus \overline{S(\infty)}$ and $n>0$ be such that $T^nJ=J$. Therefore $h\circ T^n(J) = h(J) = \{x\}$ by the Lemma~\ref{lemma:h_constant_cc}. Now $h\circ T^n(J) = T_0^n (h(J))$. Therefore $x$ is a periodic point for $T_0$, which contradicts the minimality of $T_0$.
\end{proof}

\needspace{5em}
\begin{lemma}\label{lemma:isolated_are_wandering}
The isolated points of $\overline{S(\infty)}$ are wandering points.
\end{lemma}

\begin{proof}
Let $x$ be an isolated point of $\overline{S(\infty)}$. Therefore there exists an open set $U$ such that $U \cap \overline{S(\infty)} = \{ x \}$. Hence $U \smallsetminus \{x\}= U_1 \sqcup U_2$ is included in the union of two connected components of $\gamma \smallsetminus \overline{S(\infty)}$, which are wandering sets by Lemma~\ref{lemma:permuted_without_cycle}. Therefore, $T^n(U \smallsetminus \{ x\}) \cap U \neq \emptyset$ for only finitely many values of $n$. Now, if $T^n(x) \in U$ then $T^n(x) = x$ and therefore $h(x)$ is a periodic point of $T_0$ which is impossible. Finally, we proved that $T^n U \cap U \neq \emptyset$ for only finitely many values of $n$, in other words $x$ is a wandering point.
\end{proof}

\begin{theorem}\label{thm:omega_attractor}
For every point $x \in \gamma$ whose forward orbit is infinite, the $\omega$-limit set of $x$ satisfies $\omega(x) = \Omega$. The counterpart is true for infinite backward orbits and $\alpha$-limit sets. In other words, $\Omega$ is an attractor for the transformations $T$ and $T^{-1}$. Furthermore, $\Omega$ coincides with the non-wandering set $\Omega(T)$ of $T$.
\end{theorem}

\begin{proof}
We prove both inclusions. We start by showing that $\Omega \subset \omega(x)$. By contradiction, let $y \in \Omega$ such that $y \notin \omega(x)$. Since $\omega(x)$ is a closed set, there exists an open set $U$ containing $y$ such that $U \cap \Omega \neq \emptyset$ and $U \cap \omega(x) = \emptyset$. Therefore $T^n(x) \notin U$ for large enough $n$. Since $U \cap \Omega \neq \emptyset$, $U$ contains at least two distinct points of $S(\infty)$. Since $h$ is one-to-one on $S(\infty)$ and continuous on $\gamma$, the set $h(U)$ has a non-empty interior. Therefore the sequence $T_0^n(h(x)) = h\circ T^n(x)$ is dense in $\gamma_0$ (by minimality of $T_0$) and avoids the set of non-empty interior $h(U)$, hence a contradiction.

We now prove that $\Omega^c \subset \omega(x)^c$. Let $y$ be in $\Omega^c$. There are two cases. If $y \in \gamma \smallsetminus \overline{S(\infty)}$, then by Lemma~\ref{lemma:permuted_without_cycle} $y$ is contained in a wandering interval: $y$ cannot be obtain as a limit point of an orbit by $T$, hence $y \notin \omega(x)$. Otherwise, $y$ is an isolated point of $\overline{S(\infty)}$. By contradiction, $y \in \omega(x)$ implies that $y$ is a non-wandering point, which contradicts Lemma~\ref{lemma:isolated_are_wandering}. Hence $\omega(x)=\Omega$.

We now prove that $\Omega=\Omega(T)$. By minimality of $T$ when restricted to $\Omega$, we get $\Omega \subset \Omega(T)$. Since $T$ permutes the connected components of $\gamma \smallsetminus \overline{S(\infty)}$, all points of $\gamma \smallsetminus \overline{S(\infty)}$ are wandering points. Therefore $\Omega(T) \subset \overline{S(\infty)}$. Finally, by the Lemma~\ref{lemma:isolated_are_wandering} we can refined this last inclusion by $\Omega(T) \subset \Omega$.
\end{proof}

\begin{proposition}\label{prop:Omega_K_cap_gamma}
The sets $\Omega$ and $K$ are related by $\Omega = \gamma \cap K$.
\end{proposition}

\begin{proof}
Let $p = p_i^{\sigma}$ be in $\gamma \cap K$. We know that $h_{\mathbbm{R}}(p)$ is dense in $K$, therefore $(T^n(p))_n$ is dense in $\gamma \cap K$. However, $\omega_T(x)=\Omega$ for all $x \in \gamma$, in particular for $x=p$. Hence $\Omega = \gamma \cap K$.
\end{proof}

\begin{corollary}\label{corol:ht_minimal}
When restricted to $K$, the flow $h_t$ is minimal. Furthermore, the set $K$ is an attractor for the flow $h_t$, for positive and negative times.
\end{corollary}

\begin{proof}
Let $u: \gamma \to \mathbbm{R}$ be the function giving the first return time in $\gamma$. This function is bounded by some constant $C$. Clearly, we have the equality $h_{\mathbbm{R}}(\Omega)=h_{[0,C]}(\Omega)$ and the left hand side is a closed set containing the orbit of $p = p_i^{\sigma} \in \gamma \cap K = \Omega$, hence $h_{[0,C]}(\Omega)=K$. This last equality proves the minimality of $(h_t)_t$ when restricted to $K$.

From $h_{[0,C]}(\Omega)=K$ and Theorem~\ref{thm:omega_attractor}, we obtain that every infinite forward trajectory of $h_t$ accumulates on $K$. Similarly, every infinite backward trajectory of $h_t$ accumulates on $K$.
\end{proof}

\subsection{Proof of the unique ergodicity of $h_t$}\label{sect:unique_erg_flow}

\begin{lemma}\label{lemma:factorization_measure}
In the coordinates of the suspension, every $h_t$-invariant measure $\mu$ must be of the form $\diff \mu(x,t) = C \diff\nu(x) \diff Leb(t)$, for $x \in \gamma$, $0 \leqslant t < u(x)$, some constant $C>0$ and some measure $\nu$ on $\gamma$, where $u(x)$ is the time of first return to $\gamma$ of $x$ and $Leb$ is the Lebesgue measure.
\end{lemma}

\begin{proof}
Let $\tilde{\pi} : \gamma \times \mathbbm{R} \to \mathcal{R}$ be a covering map. The lift of $h_t$ is simply the unit speed translation flow along the second coordinate. Let $\mu$ be an invariant measure for this flow. Let $\tilde{\mu}$ be a lift of $\mu$ to $\gamma \times \mathbbm{R}$. Therefore $\tilde{\mu}$ is invariant by translation along the second coordinate. Hence $\tilde{\mu} = C \nu \otimes Leb$, where $Leb$ is the Lebesgue measure and $\nu(S) \coloneqq \tilde{\mu}(S \times [0,\varepsilon])$ is a measure on $\gamma$, for some $\varepsilon >0$. Taking back the projection by $\tilde{\pi}$, we get $\diff \mu(x,t) = C \diff\nu(x) \diff Leb(t)$, as long as $\varepsilon < \inf_x u(x)$.
\end{proof}

We can now prove the unique ergodicity of $h_t$.

\begin{proof}[Proof of Theorem~\ref{thm:unique_ergodicité_(h_t)}]
Let $\mu$ be a measure invariant by the flow $h_t$. By Lemma~\ref{lemma:factorization_measure}, we can find a constant $C$ and a measure $\nu$ on $\gamma$ such that $\diff \mu(x,t) = C \diff\nu(x) \diff Leb(t)$. By applying Fubini's theorem on sufficiently small rectangles, we obtain that $\nu$ is invariant by $T$.

Since the vertical foliation associated to a pseudo-Anosov map is uniquely ergodic -- see \cite[Expos\'e~12]{fathi1979travaux, FLP_english} -- it follows that $T_0$ is uniquely ergodic. 

Now, $T$ and $T_0$ have the same path in the Rauzy-graph. By \cite{yoccoz2005echanges}, $T$ is semi-conjugated to $T_0$ by some continuous monotonic function $h$. This function $h$ is bijective when restricted, up to a countable set of points, to the set of non-wandering points of $T$. Therefore $T$ is also uniquely ergodic, of invariant measure $\nu$.
 
Hence $h_t$ is uniquely ergodic, of invariant measure $\mu$.

We now prove that the support of $\mu$ is $K$. First, since $supp(\nu)$ is included in the set of non-wandering points of $T$, which is $\Omega$, and $supp(\nu)$ is a closed set invariant by $T$, by minimality of $T$ we get that $supp(\nu) = \Omega$. Now, by the factorization of $\mu$ and the fact that $h_{\mathbbm{R}}(\Omega)=K$, we get $supp(\mu)=K$.
\end{proof}

We give in Figure~\ref{fig:oct_pert_center} a graphical representation of the set $K$ in the case of the fully explit example outlined in the description of Figure~\ref{fig:zip_rect_octagon}. More precisely, Figures~\ref{fig:zip_rect_octagon}, \ref{fig:IET_and_GIET_associated} and \ref{fig:oct_pert_center} are obtained by integrating numerically an approximation of the vector field $v^s$ associated to the pseudo-Anosov map constructed in the appendix of \cite{sinai2005weak}. The approximation of $v^s$ is obtained by truncating the sum \eqref{eq:definition_v^s} defining $v^s$. The integration is done with a forth order Runge--Kutta method. The GIET in Figure~\ref{fig:IET_and_GIET_associated} is obtained as the first return map of to $\gamma$ of \emph{many} initial conditions in $\gamma$ of the numerically estimated solutions. For Figure~\ref{fig:oct_pert_center}, we plot a numerically estimated solution with some initial data for $t \in [T/2,T]$, for some large $T>0$ (the initial point is not so important because of Corollary~\ref{corol:ht_minimal}).

As a final remark for this section, we can perform a similar analysis by perturbing a pseudo-Anosov only at some conical points $\Sigma_0 \subsetneq \Sigma$. The proofs are mostly the same by replacing $\Sigma$ by $\Sigma_0$ and using arguments from the following section in order to deal with the points $\Sigma \smallsetminus \Sigma_0 \subset K$. 

\section{Perturbation at a regular periodic point}\label{sect:perturbation_regular_point}

Because of the following general property concerning pseudo-Anosov maps -- see for example \cite{farb2011primer} -- we can consider periodic points that are not conical points -- they are regular points.
\begin{proposition}
If $\varphi : S_g \to S_g$ is pseudo-Anosov, then the set of periodic points of $\varphi$ is a dense subset of $S_g$.
\end{proposition}

Let $\theta \in S_g \smallsetminus \Sigma$ be a periodic point of $\varphi$ that is not a conical point. Up to considering a power of $\varphi$, we assume that $\theta$ is a fixed point.

In this part, we present that a very similar analysis can be done when a pseudo-Anosov map is perturbed at a fixed point that is regular instead of conical.

\subsection{Definition of the perturbation}

We can proceed to the same type of perturbation as described in Section~\ref{sect:first_def_perturb} at a regular fixed point $\theta$, except that it is much easier to define since $\theta$ is not a conical point and we do not have to deal with branched cover.

Write $\varphi(x+iy) = \lambda x + i \lambda^{-1} y$ in some local chart centred at $\theta$. In these coordinates, define \[ 
f(x+iy) \coloneqq \left( \lambda + \beta k \left( \frac{\sqrt{x^2+y^2}}{\alpha} \right) \right) x + i \lambda^{-1}y ,
\]
for some $\beta \in ]-\lambda,0]$ (we are interested in the case $\beta \in ]-\lambda,-\lambda+1[$), $0 < \alpha < \delta_\theta$, where $\delta_\theta \coloneqq \min \left( \tfrac{1}{2} Syst(S_g), \inf\{ \D(\theta, \sigma) \mid \sigma \in \Sigma \} \right)$ and $k : \mathbbm{R} \to \mathbbm{R}$ is an even unimodal function of class $\mathcal{C}^1$, compactly supported in $[-1,1]$ such that $k'$ is Lipschitz continuous, for example $k(r) = (1 - r^2)^2 \mathbbm{1}_{[-1,1]}(r)$. Set $f=\varphi$ elsewhere. Actually, this perturbation corresponds to a ``fake" conical point $\sigma = \theta$, that is where $n_\sigma =1$ (in Section~\ref{sect:construction_and_attractor} we assumed that $n_\sigma >1$ for all $\sigma \in \Sigma$).

\subsection{Differences of this case}

With the change on the range of the parameter $\alpha$, analogues of Propositions~\ref{prop:f_diffeo} to \ref{prop:K_as_Finite_Union} hold where $\sigma \in \Sigma$ is replaced by $\theta$. The proof for these results are formally the same.

That is, for $\beta \in \, ] -\lambda, 0]$ and $\alpha \in \, ]0, \delta_\theta [$, $f$ is a homeomorphism on $S_g$ and a $\mathcal{C}^1$ diffeomorphism on $S_g \smallsetminus \Sigma$. Restricting further the range of $\beta$ to be $] -\lambda, 1 - \lambda[$, $\theta$ is an attractive fixed point of $f$. Call $U_\theta$ its basin of attraction (which is open), and define $K = S_g \smallsetminus U_\theta$. In this case, we get that on the horizontal leaf containing $\theta$, there are two hyperbolic fixed points, $p_1$ and $p_2$, at the same distance $|p|>0$ from $\theta$. Furthermore $B(\theta, |p|) \subset U_\theta$. Also, $U_\theta$ is dense in $S_g$, its accessible border is $W^{ss}(p_1)\cup W^{ss}(p_2)$ and the compact set $K$ is equal to $\overline{W^{ss}(p_1)} \cup \overline{W^{ss}(p_2)}$.

The first difference with the case treated in the previous sections is that here $\Sigma \subset K$. In particular, $K$ can no longer be a hyperbolic set since $f$ is not differentiable on points of $\Sigma$. Nonetheless, one can still construct a vector field $v^s$ on $K \smallsetminus \Sigma$ such that $\mathrm{d}_x f (v^s(x)) = \lambda^{-1} v^s(f(x))$, $x \in K \smallsetminus \Sigma$, as in Theorem~\ref{thm:K_is_hyperbolic} (in particular, strictly speaking, $f$ is not Axiom A since $K \smallsetminus \Sigma$ is not a compact set). Results from the first half of Section~\ref{sect:technicalities} do not change, as well as there proofs. More precisely, $v^s$ can be extended into a Lipschitz (or $\mathcal{C}^1$ for smooth $k$) vector field on $S_g \smallsetminus \Sigma$, such that $(x,\beta) \in (S_g \smallsetminus \Sigma) \times ]-\lambda + \lambda^{-2}, 0] \mapsto v^s_\beta(x)$ is continuous.

The second difference concerns Proposition~\ref{prop:commutation_relation_and_stability_of_K}. The commutation relation between $f$ and $h_t$ still holds, but $K$ is no longer $h_t$-invariant -- this is due to the fact that each branch of stable manifold of each $\sigma \in \Sigma$ belongs to $K$ and is parametrized by $h_t$, thus some trajectories in $K$ end in finite time.

Nonetheless, analogue of Proposition~\ref{prop:stable_leaf_is_flow_trajectory} and Theorems~\ref{thm:K_connected} and \ref{thm:f_transitive_on_K} hold. More precisely, only the statement of Theorem~\ref{thm:f_transitive_on_K} has to be modified: $f : K \smallsetminus \Sigma \to K \smallsetminus \Sigma$ is topologically transitive for the trace topology of $S_g$ on $K \smallsetminus \Sigma$. Proofs do not change.

Concerning Section~\ref{sect:f_is_mixing}, since the study of GIET and IET already deal with the singularities of the flows at $\Sigma$, the results concerning the map $T$ and $T_0$ have analogous counterparts. Here the map segments $\gamma$ and $\gamma_0$ are contained in the horizontal leaf containing $\theta$. In particular there is a subset $\Omega \subset \gamma$, attractor for positive and negative iterates of $T$, such that $T|_\Omega$ is minimal, and $T$ is uniquely ergodic of unique invariant measure $\nu$ supported by $\Omega$. The relation $\Omega = \gamma \cap K$ still holds with change in the proof. Since the measure $\nu$ has no atoms -- otherwise, $T_0$ would have periodic orbits, contradicting the minimality of $T_0$ -- we get that $\nu(\Omega) = \nu(\Omega \smallsetminus S(\infty)) = 1$. Furthermore, since $\Omega \smallsetminus S(\infty) = \gamma \cap (K \smallsetminus \cup_{i \in \{ 1,2 \}} W^{ss}(p_i))$, and $K \smallsetminus \cup_{i \in \{ 1,2 \}} W^{ss}(p_i)$ is $h_t$-invariant, we get that $\mu = C \nu \otimes Leb$ is a $h_t$-invariant probability measure (for some $C>0$), with support $\overline{K \smallsetminus \cup_{i \in \{ 1,2 \}} W^{ss}(p_i)}= K$. 

Finally, $h_t$ is uniquely ergodic, and $f$ is mixing with respect to $\mu$, as in the previous section.

\begin{figure}
\begin{center}
\raisebox{-0.5\height}{\includegraphics[width=0.4\textwidth]{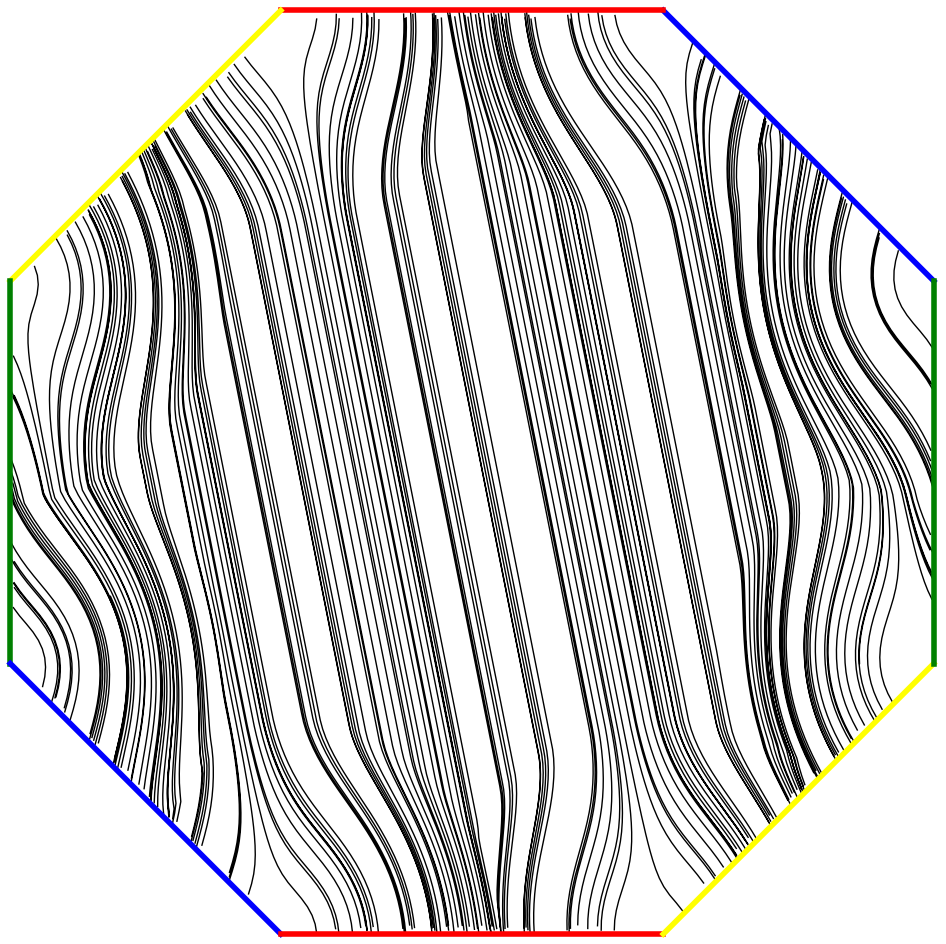}} \hspace{1cm}
\raisebox{-0.5\height}{\includegraphics[width=0.4\textwidth]{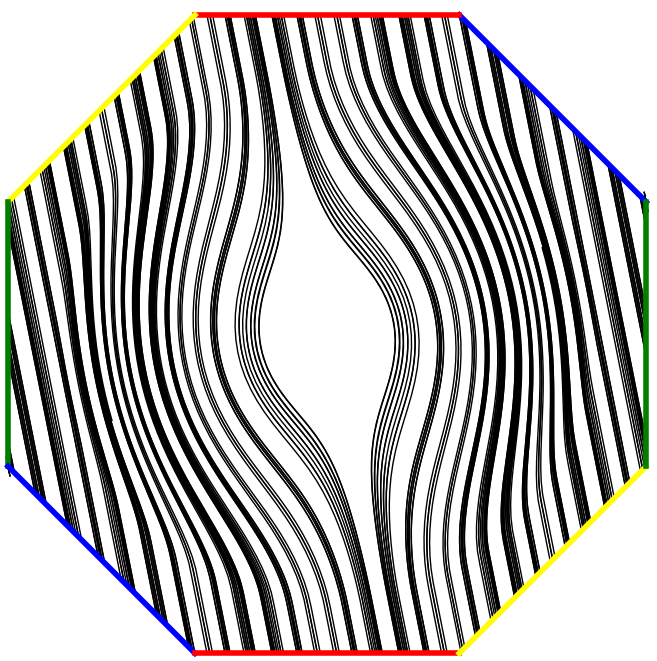}}
\end{center}
\caption{\label{fig:oct_pert_center} Numerical representations of the set $K$ for a perturbation of a pseudo-Anosov homeomorphism on a genus two surface. \textsc{Right:} perturbation at the unique conical point. \textsc{Left:} perturbation at a regular fixed point (the center of the octagon).}
\end{figure}

\section{The measure $\mu$}

In this last section, using extensively Bowen and Ruelle's work \cite{bowen2008equilibrium, Ruelle1976measure}, we prove that $\mu$ is the unique SRB-like measure of $f^{-1}$, and that correlations decrease exponentially fast for $\mathcal{C}^1$ observables compactly supported away from $\Sigma$, in the case where $f$ is constructed as in Section~\ref{sect:construction_and_attractor}. Finally, using the maximizing property associated with SRB measure, we compute the entropy of $f$ with respect to $\mu$. We also ask whether the result on the Ruelle spectrum of a linear pseudo-Anosov by Faure, Gou\"ezel and Lanneau \cite{faure2019ruelle}, and the asymptotic expansion for ergodic integral of the Giulietti--Liverani flow proved by Forni \cite{Forni2020equidistribution}, can be adapted to the settings of the present paper.

We used the term ``SRB-like" instead of just ``SRB" because SRB measure are only defined for $\mathcal{C}^2$ (or $\mathcal{C}^{1+ \alpha}$) diffeomorphisms, but the above map $f$ is only continuous at conical points. Nonetheless, we show that $\mu$ is the unique SRB measure associated to $f^{-1} \vert_{S_g \smallsetminus \Sigma}$ and that the usual definitions of SRB measure extend to $f^{-1}$. We will therefore refer to SRB measure in the rest of this section instead of ``SRB-like" measure.

For now on, we assume that $f$ is a $\mathcal{C}^2$ diffeomorphism away from $\Sigma$, which can be achieved by choosing a $\mathcal{C}^2$ bump function $k$. Such a bump function $k$ is also assumed to be $\mathcal{C}^2$.

\subsection{SRB measure and entropy of $f^{-1}$}

Sinai--Ruelle--Bowen measures are particular invariant measures of $\mathcal{C}^2$ transformations. See \cite{young2002srb} for a survey about these measures and which dynamical systems have them. 

The problem here is that $f$ and $f^{-1}$ are smooth only away from conical points, where they are only continuous. Still, $S_g \smallsetminus \Sigma$ is an invariant set on which $f^{-1}$ is a $\mathcal{C}^2$ diffeomorphism. Furthermore, $K$ is an Axiom A attractor for $f^{-1}$, in the sense that $K$ is locally maximal, $f^{-1}\vert_K$ is uniformly hyperbolic and $f^{-1} \vert_K$ is topologically transitive. Notice that $K$ is connected.

By \cite[Theorem 1.5]{Ruelle1976measure}, there exists a unique SRB measure $\mu_K$ supported by $K$, maximizing $h_{\nu}(f^{-1}\vert_{S_g \smallsetminus \Sigma}) + \nu( - \log \det \mathrm{d}f^{-1}\vert_{E^s} )$ -- and the maximum is equal to $0$.

\begin{theorem}\label{thm:SRB_measure}
If $W$ is a curve of finite length contained in $W^{ss}(p)$ and containing $p$, where $p$ is some hyperbolic fixed point $p^{\sigma}_i$ of $f$, and $\nu_W$ is a measure on $W$ with bounded Radon-Nikodym derivative with respect to the measure induce by the Riemann metric on $W$, then $\mu = \lim\limits_{n \to \infty} (f^{-n})_* \nu_W$.

In particular, $\mu = \lim\limits_{N \to \infty} \frac{1}{N} \sum\limits_{n=0}^{N-1} f^{-n}_* \nu_W$ and according to \cite{young2002srb}, $\mu$ is a SRB measure for $f^{-1}\vert_{S_g \smallsetminus \Sigma}$. Therefore, by uniqueness, $\mu = \mu_K$.
\end{theorem}

\begin{proof}
Let $W \subset \tilde{W} \subset W^{ss}(p)$ be a strictly longer curve than $W$. Let $\tilde{\nu}$ be the measure on $\tilde{W}$ induced by the Riemann metric. By assumption there exists a bounded function $\rho \geqslant 0$ such that $\mathrm{d}\nu_W = \rho \, \mathrm{d}\tilde{\nu}$. If needed, $\rho$ is implicitly extended by $0$.

Since $k$ is assumed to be $\mathcal{C}^2$, by Theorem~\ref{thm:vs_C1}, $h_t$ is a $\mathcal{C}^1$ flow. Therefore, for small enough $t$, $(h_t)_* \nu_W$ is supported by $\tilde{W}$ and $$ \mathrm{d}((h_t)_*\nu_W) = \frac{\rho}{\mathrm{Jac}\, h_t} \circ h_{-t} \, \mathrm{d}\tilde{\nu},$$
Where $\mathrm{Jac \, h_t}$ is the Jacobian determinant of the time $t$ of the flow. 
Therefore, if $\varphi$ is a continuous function on $S_g$, then for all small enough $t$,
\begin{align*}
|(h_t)_*(f^{-n}_* \nu_W)  - (f^{-n}_* \nu_W) |(\varphi) 
&= 
|f^{-n}_*( (h_{\lambda^{-n}t})_* \nu_W - \nu_W)|(\varphi) \\
&\leqslant
|\varphi|_{\infty} \int_{\tilde{W}} \left| \frac{\rho}{\mathrm{Jac}\, h_{\lambda^{-n}t}}\circ h_{-\lambda^{-n}t} - \rho \right| \, \mathrm{d}\tilde{\nu},
\end{align*}
which converges, by dominated converge, to zero as $n$ goes to infinity. Therefore, all subsequential limits of $f^{-n}_* \nu_W$ are $h_t$-invariant. By unique ergodicity of $h_t$, all subsequential limits of $f^{-n}_* \nu_W$ must coincide with $\mu$. Therefore $f^{-n}_* \nu_W$ converges to $\mu$.
\end{proof}

We can now compute the entropy of $f$ with respect to $\mu$.

\begin{theorem}
The entropy $h_{\mu}(f)$ with respect to $\mu$ is equal to $\log(\lambda)$.
\end{theorem}

\begin{proof}
It follows from the fact that $\mathrm{d}f \, v^s = \lambda^{-1} v^s \circ f$ that $\mathrm{d}f^{-1}\vert_{E^s}$ is constant equal to $\lambda$ on $K$. Therefore $h_{\mu}(f^{-1}\vert_{S_g \smallsetminus \Sigma}) = \log( \lambda)$. Now, since $\Sigma \cap K = \emptyset$, we get that $h_{\mu}(f)= h_{\mu}(f^{-1}) = \log(\lambda)$.
\end{proof}

Finally, remark that since the nonwandering set of $f$ is $K \cup \Sigma$ and since we can extend by continuity $\mathrm{d}f^{-1}\vert_{E^s}$ at each $\sigma$ in $\Sigma$ by $\lambda^{-1}(\lambda + \beta_{\sigma})<1$, the measure $\mu$ is still the unique measure maximizing $h_{\nu}(f^{-1}) + \nu( - \log \det \mathrm{d}f^{-1}\vert_{E^s} )$ for $\nu$ ranging over the set of $f$-invariant measures.

\subsection{Bernoulli and exponential mixing}

Using the careful analysis over Markov partition done by Ruelle in \cite{Ruelle1976measure}, we are able to deduce that $(f,\mu)$ is isomorphic to a Bernoulli shift and that the correlations decrease exponentially fast for $\mathcal{C}^1$ observables supported away from $\Sigma$. 

\begin{theorem}
The system $(f,\mu)$ is isomorphic to a Bernoulli shift.
\end{theorem}

\begin{theorem}
There exist constants $0< \theta < 1$ and $C >0$ such that for all $\mathcal{C}^1$ observables $\varphi$ and $\psi$ compactly supported away from $\Sigma$, 
$$|\mu(\varphi \circ f^{-n} \, \psi) - \mu(\varphi) \mu(\psi)| < C ||\varphi||_{\mathcal{C}^1} ||\psi||_{\mathcal{C}^1} \theta^{-n}, \quad \forall n \geqslant 0.$$
\end{theorem}

The proofs of these two theorems directly follows from \cite[Theorem 1.5]{Ruelle1976measure}.

\subsection{What about the Ruelle spectrum?}

In \cite{faure2019ruelle}, Faure, Gou\"ezel and Lanneau proved that for any orientation preserving linear pseudo-Anosov map $\varphi$ on a surface $S_g$ of genus $g$, the Ruelle spectrum can be computed explicitly. More precisely, if $\lambda > 1$ is the expansion factor of $\varphi$ and $\lambda^{-1},\, \lambda,\, \mu_1, \ldots,\, \mu_{2g-2}$ is the spectrum of $\varphi^*$ --~where $\varphi^*$ is the natural action of $\varphi$ on the first space of cohomology $H^1(S_g)$~-- then the Ruelle spectrum of $\varphi$ for $\mathcal{C}_c^{\infty}(S_g \smallsetminus \Sigma)$ observables is $\{ \lambda^{-n}\mu_i \mid 1 \leqslant i \leqslant 2g-2, \, n \geqslant 1 \}$. Furthermore, the multiplicity of $\lambda^{-n}\mu_i$ is $n$. In order to prove this result, the authors first show that $\lambda^{-n}\mu_i$ are indeed Ruelle resonances and then that there are no other Ruelle resonances. 

Since $f$ is, by construction, homotopic to such linear pseudo-Anosov map $\varphi$, the action on the cohomology is the same. One might expect that the Ruelle spectrum of $(f,\mu)$ is the same as the one of $\varphi$, up to a few modifications. 

The key ingredients in the first part of \cite{faure2019ruelle} --~where it is proved that $\lambda^{-n}\mu_i$ are Ruelle resonances~-- are the smoothness of the invariant foliations and the uniform contraction of the stable foliation. This particularities remain true in the case of the perturbation $f$. The argument then should carry over to the case of the specific derived from pseudo-Anosov maps studied in this paper.

However, the second part of \cite{faure2019ruelle} -- where it is proved that Ruelle resonances must be of the form $\lambda^{-n}\mu_i$ -- relies on many geometric considerations and also on the uniform dilation of the unstable foliation. Unfortunately this last assumption fails, by construction, in the case of $f$. 

\subsection{What about deviation from Ergodic Integral?} In \cite[Corollary 1.5]{Forni2020equidistribution}, Forni proved an asymptotic expansion for the ergodic integrals of the Giulietti--Liverani flow \cite{GL_2019} on surface of genus $g \geqslant 2$. Because of all the common properties between the flow $h_t$ studied in the present paper and the Giulietti--Liverani flow, it seems reasonable that a similar formula should holds. However, it is not clear whether Forni's proof can be adapted in this setting.

\bibliography{biblioARPE}{}

\begin{thebibliography}{10}

\bibitem{athanassopoulos2015denjoy}
K.~Athanassopoulos.
\newblock Denjoy {$C^1$} diffeomorphisms of the circle and {M}c{D}uff's
  question.
\newblock {\em Expo. Math.}, 33(1):48--66, 2015.

\bibitem{Barge11classification}
M.~Barge and B.~F. Martensen.
\newblock Classification of expansive attractors on surfaces.
\newblock {\em Ergodic Theory Dynam. Systems}, 31(6):1619--1639, 2011.

\bibitem{Bonatti1998}
C.~Bonatti and R.~Langevin.
\newblock Diff\'{e}omorphismes de {S}male des surfaces.
\newblock {\em Ast\'{e}risque}, (250):viii+235, 1998.
\newblock With the collaboration of E. Jeandenans.

\bibitem{bowen2008equilibrium}
R.~E. Bowen.
\newblock {\em Equilibrium states and the ergodic theory of {A}nosov
  diffeomorphisms}, volume 470.
\newblock Springer Science \& Business Media, 2008.

\bibitem{hubert2010persistence}
X.~Bressaud, P.~Hubert, and A.~Maass.
\newblock Persistence of wandering intervals in self-similar affine interval
  exchange transformations.
\newblock {\em Ergodic Theory Dynam. Systems}, 30(3):665--686, 2010.

\bibitem{Butterley2020parabolic}
O.~Butterley and L.~D. Simonelli.
\newblock Parabolic flows renormalized by partially hyperbolic maps.
\newblock {\em Boll. Unione Mat. Ital.}, 13(3):341--360, 2020.

\bibitem{camelier1997aiet}
R.~Camelier and C.~Gutierrez.
\newblock Affine interval exchange transformations with wandering intervals.
\newblock {\em Ergodic Theory Dynam. Systems}, 17(6):1315--1338, 1997.

\bibitem{coudene2006pictures}
Y.~Coud\`ene.
\newblock Pictures of hyperbolic dynamical systems.
\newblock {\em Notices of the AMS}, 53(1), 2006.

\bibitem{coudene2013book}
Y.~Coud\`ene.
\newblock {\em Ergodic theory and dynamical systems}.
\newblock Universitext. Springer-Verlag London, Ltd., London; EDP Sciences,
  [Les Ulis], 2016.
\newblock Translated from the 2013 French original [ MR3184308] by Reinie
  Ern\'{e}.

\bibitem{farb2011primer}
B.~Farb and D.~Margalit.
\newblock {\em A primer on mapping class groups (pms-49)}.
\newblock Princeton University Press, 2011.

\bibitem{FLP_english}
A.~Fathi, F.~Laudenbach, and V.~Po\'{e}naru.
\newblock {\em Thurston's work on surfaces}, volume~48 of {\em Mathematical
  Notes}.
\newblock Princeton University Press, Princeton, NJ, 2012.
\newblock Translated from the 1979 French original by Djun M. Kim and Dan
  Margalit.

\bibitem{fathi1979travaux}
A.~Fathi, F.~Laudenbach, V.~Po{\'e}naru, et~al.
\newblock Travaux de {T}hurston sur les surfaces, volume 66-67 of
  {A}st{\'e}risque.
\newblock {\em Soci{\'e}t{\'e} Math{\'e}matique de France, Paris}, 1979.

\bibitem{faure2019ruelle}
F.~Faure, S.~Gou{\"e}zel, and E.~Lanneau.
\newblock Ruelle spectrum of linear pseudo-{A}nosov maps.
\newblock {\em Journal de l'{\'E}cole polytechnique-Math{\'e}matiques},
  6:811--877, 2019.

\bibitem{Forni2020equidistribution}
G.~Forni.
\newblock On the equidistribution of unstable curves for pseudo-{A}nosov
  diffeomorphisms of compact surfaces.
\newblock {\em Ergodic Theory Dynam. Systems}, 42(3):855--880, 2022.

\bibitem{forni2013introduction}
G.~Forni and C.~Matheus.
\newblock Introduction to {T}eichm\"{u}ller theory and its applications to
  dynamics of interval exchange transformations, flows on surfaces and
  billiards.
\newblock {\em J. Mod. Dyn.}, 8(3-4):271--436, 2014.

\bibitem{GL_2019}
P.~Giulietti and C.~Liverani.
\newblock Parabolic dynamics and anisotropic {B}anach spaces.
\newblock {\em J. Eur. Math. Soc. (JEMS)}, 21(9):2793--2858, 2019.

\bibitem{herman1979conjugaison}
M.-R. Herman.
\newblock Sur la conjugaison diff\'{e}rentiable des diff\'{e}omorphismes du
  cercle \`a des rotations.
\newblock {\em Inst. Hautes \'{E}tudes Sci. Publ. Math.}, (49):5--233, 1979.

\bibitem{hubbard2016teichmuller}
J.~H. Hubbard.
\newblock {\em Teichm{\"u}ller theory and applications to geometry, topology,
  and dynamics}.
\newblock 2016.

\bibitem{katok1997introduction}
A.~Katok and B.~Hasselblatt.
\newblock {\em Introduction to the modern theory of dynamical systems},
  volume~54.
\newblock Cambridge university press, 1997.

\bibitem{lanneau2017tell}
E.~Lanneau.
\newblock Tell me a pseudo-{A}nosov.
\newblock {\em Eur. Math. Soc. Newsl.}, (106):12--16, 2017.
\newblock Translated from the French [ MR3643215] by Fernando P. da Costa.

\bibitem{levitt1987feuilletages}
G.~Levitt.
\newblock La d\'{e}composition dynamique et la diff\'{e}rentiabilit\'{e} des
  feuilletages des surfaces.
\newblock {\em Ann. Inst. Fourier (Grenoble)}, 37(3):85--116, 1987.

\bibitem{MMY2010aiet}
S.~Marmi, P.~Moussa, and J.-C. Yoccoz.
\newblock Affine interval exchange maps with a wandering interval.
\newblock {\em Proc. Lond. Math. Soc. (3)}, 100(3):639--669, 2010.

\bibitem{Rodriguez06expansive}
F.~Rodriguez~Hertz and J.~Rodriguez~Hertz.
\newblock Expansive attractors on surfaces.
\newblock {\em Ergodic Theory Dynam. Systems}, 26(1):291--302, 2006.

\bibitem{Ruelle1976measure}
D.~Ruelle.
\newblock A measure associated with axiom-{A} attractors.
\newblock {\em Amer. J. Math.}, 98(3):619--654, 1976.

\bibitem{sinai2005weak}
Y.~G. Sinai and C.~Ulcigrai.
\newblock Weak mixing in interval exchange transformations of periodic type.
\newblock {\em Letters in Mathematical Physics}, 74(2):111--133, 2005.

\bibitem{smale1967differentiable}
S.~Smale.
\newblock Differentiable dynamical systems.
\newblock {\em Bulletin of the American mathematical Society}, 73(6):747--817,
  1967.

\bibitem{Williams70DA}
R.~F. Williams.
\newblock The {$``{\rm DA}''$} maps of {S}male and structural stability.
\newblock In {\em Global {A}nalysis ({P}roc. {S}ympos. {P}ure {M}ath., {V}ol.
  {XIV}, {B}erkeley, {C}alif., 1968)}, pages 329--334. Amer. Math. Soc.,
  Providence, R.I., 1970.

\bibitem{yoccoz2005echanges}
J.-C. Yoccoz.
\newblock Echanges d'intervalles.
\newblock \textit{Cours Coll{\`e}ge de France},
  \url{https://www.college-de-france.fr/media/jean-christophe-yoccoz/UPL8726_yoccoz05.pdf},
  2005.

\bibitem{yoccoz2010survey}
J.-C. Yoccoz.
\newblock Interval exchange maps and translation surfaces.
\newblock In {\em Homogeneous flows, moduli spaces and arithmetic}, volume~10
  of {\em Clay Math. Proc.}, pages 1--69. Amer. Math. Soc., Providence, RI,
  2010.

\bibitem{young2002srb}
L.-S. Young.
\newblock What are {SRB} measures, and which dynamical systems have them?
\newblock {\em Journal of Statistical Physics}, 108(5-6):733--754, 2002.

\bibitem{zorich:hal-00453397}
A.~Zorich.
\newblock {Flat surfaces}.
\newblock In {\em {Frontiers in Number Theory, Physics, and Geometry I}}, pages
  439--585. {Springer}, 2006.

\end{thebibliography}
\bibliographystyle{abbrv}

\end{document}